\documentclass{amsart}
\usepackage{amssymb,amsmath,epsfig,graphics,latexsym,psfrag}
\usepackage[english]{babel}
\usepackage[all]{xy}
\usepackage{amscd, amssymb, amsmath, amsthm, graphics}
\usepackage{amsmath,amsfonts,amsthm,amssymb}
\usepackage{latexsym,amsmath}
\usepackage{graphicx,psfrag}
\usepackage{mathrsfs}

\newtheorem{theorem}{Theorem}[section]
\newtheorem{proposition}[theorem]{Proposition}
\newtheorem{corollary}[theorem]{Corollary}
\newtheorem{lemma}[theorem]{Lemma}

\theoremstyle{definition}
\newtheorem{definition}[theorem]{Definition}
\newtheorem{example}[theorem]{Example}

\theoremstyle{remark}
\newtheorem{remark}[theorem]{Remark}
\newtheorem{problem}[theorem]{Problem}
\newtheorem{question}{Question}
\numberwithin{equation}{section}
\renewcommand{\t}{ \widetilde}
\renewcommand{\hat}{ \widehat}
\renewcommand{\b}{ \partial}
\newcommand{\Z}{\mathbb Z}
\newcommand{\R}{\mathbb R}
\newcommand{\N}{\mathbb N}
\newcommand{\C}{\mathbb C}

\newcommand{\Hi}{\bf H}

\newcommand{\RR}{{\mathbb R}}

\newcommand{\ZZ}{{\mathbb Z}}

\newcommand{\mobgp}{{\mathrm{PSL}_2(\mathbb{C})}}

\newcommand{\ud}{{\mathrm{d}}}

\renewcommand{\S}{\bf S}
\renewcommand{\l}{\langle}
\renewcommand{\r}{\rangle}

\newcommand{\z}[1]{{\Z}/#1{\Z}}
\renewcommand{\o}{\overline}

\newcommand{\co}{\colon\thinspace}
\renewcommand{\epsilon}{\varepsilon}
\renewcommand{\c}{\mathcal}

\usepackage{times}
\usepackage{times}
\begin{document}
\sloppy

\title[Chern--Simons, separability, and volumes]{Chern--Simons theory, surface separability,  \\and  
volumes of 3-manifolds}

\author{Pierre Derbez}
\address{LATP UMR 7353, 
39 rue  Joliot-Curie,
 13453 Marseille Cedex 13}
\email{pderbez@gmail.com}

\author[Yi Liu]{Yi Liu}
\address{
    Mathematics 253-37\\
    California Institute of Technology\\
    Pasadena, CA 91125, U.S.A.}
\email{yliumath@caltech.edu}
\author{Shicheng Wang}
\address{Department of Mathematics, Peking University, Beijing, China}
\email{wangsc@math.pku.edu.cn}



\date{\today}

\begin{abstract}
We  study  the set ${\rm vol}\left(M,G\right)$  of  volumes  of all
representations $\rho\co\pi_1M\to G$, where $M$ is a closed oriented
$3$-manifold and $G$ is either ${\rm
Iso}_+{\Hi}^3$ or ${\rm Iso}_e\t{\rm SL_2(\R)}$.

By various methods, including relations between the
volume of representations and the Chern--Simons invariants of flat
connections, and recent results of surfaces in 3-manifolds, we prove that any 3-manifold $M$ with positive Gromov simplicial volume
has a finite cover  $\t M$ with ${\rm vol}(\t
M,{\rm
Iso}_+{\Hi}^3)\ne \{0\}$, and that any non-geometric 
3-manifold $M$ containing at least one  Seifert  piece has a finite cover  $\t M$ with ${\rm vol}(\t M,{\rm Iso}_e\t{\rm SL_2(\R)}) \ne \{0\}$.  

We also find 
3-manifolds  $M$ with positive simplicial volume but ${\rm
vol}(M,{\rm
Iso}_+{\Hi}^3)=\{0\}$,  and non-trivial graph manifolds $M$ with 
${\rm
vol}(M,{\rm Iso}_e\t{\rm SL_2(\R)})=\{0\}$,
 proving that it is in general necessary to pass to some finite covering to guarantee that ${\rm vol}(M,G)\not=\{0\}$. 

Besides we determine ${\rm vol}\left(M, G \right)$ when $M$ supports the
Seifert geometry.

\end{abstract}

\maketitle

\vspace{-.5cm}
\tableofcontents

\section{Introduction}

The volume of representations of 3-manifolds groups is a beautiful
theory which has  rich connections with many branches of
mathematics. However the behavior of those volume functions seems
still quite mysterious. To make our meaning more explicit, we first
give some basic notions (which will be defined later) and properties
of volume of representations. Let $N$ be  a closed orientable
$3$-manifold. Let $G$ be either ${\rm
Iso}_+{\Hi}^3\cong{\rm PSL}(2;{\C})$, the orientation preserving isometry group of the
hyperbolic 3-space, or ${\rm Iso}_e\t{\rm SL_2(\R)}\cong\R\times_\Z\t{\rm SL_2(\R)}$, the identity
component of the isometry group of  $\t{\rm SL_2(\R)}$. For each
representation $\rho\co\pi_1M\to G$, the volume of
 $\rho$ is denoted by ${\rm vol}_G(M,\rho)$. 
 
 Define
$${\rm vol}\left(M,G\right)\,=\,\left\{{\rm vol}_G(M,\rho) \textrm{when}\ \rho\ \textrm{runs over the representations} \pi_1M\to G\right\}$$
Suppose $M$ supports a hyperbolic, respectively an $\t{\rm SL_2(\R)}$-geometry.
Then $M$ naturally has its own hyperbolic volume ${\rm
vol}_{{\Hi}^3}(M)$, respectively Seifert volume ${\rm vol}_{\t{\rm
SL_2(\R)}}(M)$. We denote by $||M||$  the Gromov norm of $M$,
which measures, up to a multiplicative constant, the total hyperbolic
volume of the hyperbolic pieces of $M$ \cite{Gr,So}. The following
theorem contains some known basic results of the theory of volume
representations. For its development, see \cite{BG1,BG2,Re-rationality} and their references.

\begin{theorem}\label{basic property of volume of presentation}
	Let $N$ be a closed orientable $3$-manifold.
	\begin{enumerate}
		\item Both ${\rm vol}(N,{\rm PSL}(2;{\C}))$ and ${\rm vol}(N, {\rm
		Iso}_e\t{\rm SL_2(\R)})$ contain at most finitely many values. 
		Hence the supremums ${\mathrm{HV}}(N)$ and ${\mathrm{SV}}(N)$ of ${\rm vol}(N,{\rm PSL}(2;{\C}))$ and ${\rm vol}(N, {\rm
		Iso}_e\t{\rm SL_2(\R)})$ are reached.
		\item If $N$ admits a hyperbolic geometric structure, then ${\mathrm{HV}}(N)$ equals
		${\rm vol}_{{\Hi}^3}(N)$, reached by any discrete and faithful representation. 
		A similar statement holds when $N$ admits an $\t{\rm SL_2(\R)}$ geometric structure.
		\item ${\mathrm{HV}}(N)\leq \mu_3\,||N||$, 
		where $\mu_3$ denotes the volume of any ideal regular tetrahedron in ${\Hi}^3$.
		\item Let $f\co M\to N$ be a map between closed orientable $3$-manifolds 
		and let $\rho: \pi_1 N\to G$
		denote a representation. Then 
			$${\rm vol}_G(M,f^*\rho)= \mathrm{deg}(f)\,{\rm vol}_G(N,\rho).$$
		Hence, $${\mathrm{HV}}(M)\geq|{\rm deg}f|\,{\mathrm{HV}}(N)\textrm{ and } {\mathrm{SV}}(M)\geq|{\rm deg}f|\,{\mathrm{SV}}(N).$$
	\end{enumerate}
\end{theorem}

We call ${\mathrm{HV}}(N)$ and ${\mathrm{SV}}(N)$ in the conclusion of Theorem \ref{basic property of volume of presentation}
(1) the \emph{hyperbolic volume} and the \emph{Seifert volume} of $N$, respectively.

\begin{remark} \label{relation with $D(M,N)$}  
Let $M$ and $N$ be two closed oriented $3$-dimensional
manifolds and let ${D}(M,N)$ be the set of degrees of maps from $M$
to $N$. Let 
${\mathcal D}$ be the set of all closed orientable 3-manifolds $N$ with $D(M,N)$
finite for any fixed $M$. 
By Theorem \ref{basic property of volume of
presentation} (4), ${\mathrm{SV}}(N)={\mathrm{HV}}(N)=0$ if $N\notin \mathcal D$. 
It is known that (see \cite{DSW} for example),  
$N\in \mathcal D$ 
if and only if $N$ contains a prime factor $Q$  with non-trivial geometric decomposition,
or supporting an $\t{\rm SL_2(\R)}$ or a  hyperbolic geometry.
This fact combined with  Theorem \ref{basic property of volume of
presentation} (2), (3), (4)   implies that if  ${\rm vol}\left(N,
\rm {Iso}_e\t{\rm SL_2(\R)}\right)\not=\{0\}$ then necessarily
a prime factor of $N$ has a non-trivial geometric decomposition, or  
supports an $\t{\rm SL_2(\R)}$  or a hyperbolic geometry and if ${\rm
vol}\left( N, \rm PSL(2;{\C}) \right)\not=\{0\}$ then necessarily
a prime factor of $N$ contains some hyperbolic JSJ pieces.
\end{remark}

Besides Theorem \ref{basic property of volume of presentation},
Thurston pointed out the relation between Chern--Simons invariants
and the hyperbolic volume of hyperbolic 3-manifolds for discrete
and faithful representations \cite{Th2}. Such a relation is
extended by Kirk--Klassen \cite{KK} for cusped hyperbolic 3-manifolds and discrete
and faithful representations into ${\rm PSL}(2;{\C})$, and by Khoi \cite{Kh} for closed manifolds
with the group $\t{{\rm SL}_2(\R)}$ (as a subgroup of ${\rm Iso}\t{\rm
SL_2(\R)}$).

Despite those significant results, the answer to the questions below, which is
a main motivation of  this paper, seems still remarkably unknown.
Recall that a non-negative invariant $\eta$ of 3-manifolds 
is said to satisfy the \emph{covering
property} in the sense of Thurston, 
if for any finite covering $p: \t N \to N$, we have
$\eta(\t N)= |\mathrm{deg}(p)|  \eta(N)$.

\begin{question}\label{Q1}
	Let $M$ be a closed $3$-manifold and let $G$ be either ${\rm PSL}(2,{\C})$ or ${\rm Iso}_e\t{{\rm SL}_2(\R)}$.
	\begin{enumerate}
		\item 
		\begin{enumerate}
			\item How to find  non-zero elements in ${\rm vol}(M, G)$?
			\item More weakly, how to find  non-zero elements in ${\rm vol}(\t M, G)$
			for some finite cover $\t M$ of $M$?
		\end{enumerate}
		\item Does ${\mathrm{HV}}$ or ${\mathrm{SV}}$ satisfy the covering property?
	\end{enumerate}


\end{question}

\begin{remark}
Three-manifold invariants with the
covering property was first addressed by Thurston in the 1970s \cite[Problem
3.16(A)]{Ki}.  The simplicial volume has the covering property
(See Gromov, Thurston, Soma \cite{Gr,Th1,So}, an early evidence  of such application appears in Milnor--Thurston \cite{MT}). 
Some papers define invariants with
the covering property for graph manifolds, say \cite{WW,LW,Ne}, but each one vanishes
for some graph manifolds.

So far it seems that we only know that ${\mathrm{HV}}$, respectively, ${\mathrm{SV}}$, satisfies
the covering property for the hyperbolic, respectively, Seifert manifolds.
In hyperbolic geometry this property comes from the relation between
the simplicial volume and ${\mathrm{HV}}$. In Seifert geometry one can compute
${\mathrm{SV}}$ in terms of the Euler
 classes of the Seifert manifold and the Euler characteristic of its orbifold   and these
invariants behave naturally under covering maps.
\end{remark}


The results of Brooks--Goldman  \cite{BG2} and  Eisenbud--Hirsch--Neumann  \cite{EHN}  
allow us to
describe the set ${\rm vol}\left(M, \rm {Iso}_e\t{\rm
SL_2(\R)}\right) $ for each closed 3-manifold $M$ supporting an $\t{\rm 
SL_2(\R)}$-geometry.
It is known that $M$ supports an $\t{\rm SL_2(\R)}$-geometry if and
only if  $M$ is a Seifert manifold with non-zero Euler number $e(M)$
over an  orbifold of negative Euler characteristic. We use
$\llcorner{a }\lrcorner$ and  $\ulcorner a\urcorner$ for $a\in{\R}$
to denote respectively, the greatest integer less than or equal to 
$a$ and the least integer greater than or equal to $a$.

\begin{proposition}\label{Seifert volumes of Seifert manifolds}
Suppose  $M$ supports an $\t{\rm SL_2(\R)}$-geometry and that its
base $2$-orbifold has a positive genus $g$. Then
\begin{eqnarray}
{\rm vol}\left(M, \rm {Iso}_e\t{\rm SL_2(\R)}
\right)=\left\{\frac{4\pi^2}{|e(M)|}\left(\sum_{i=1}^r\left(\frac{n_i}{a_i}\right)-n\right)^2\right\}
\end{eqnarray}
where $n_1,\ldots,n_r,n$ are integers such that $$\sum_{i=1}^r
\llcorner{ {n_i}/{a_i}}\lrcorner -n\le 2g-2, \,\,\, \sum_{i=1}^r
\ulcorner{n_i}/{a_i}\urcorner-n \ge 2-2g$$ and $a_1, \ldots, a_r$ are
 the indices of the singular points of the orbifold of $M$.
\end{proposition}

\begin{remark}  In order to check Proposition \ref{Seifert volumes of Seifert
manifolds},  we will describe all representations with non-zero
volume. 
Proposition \ref{Seifert volumes of Seifert manifolds} presents
explicitly the rationality of the elements in ${\rm vol}(M, \rm
{Iso}_e\t{\rm SL_2(\R)})$, which was proved by Reznikov \cite{Re-rationality}.
\end{remark}


As a partial answer to Question \ref{Q1} (1) for non-geometric manifolds,
it was known  recently that each non-trivial graph manifold $M$
has a finite cover $\t M$ such that ${\rm vol}(\t M, {\rm
Iso}_e\t{\rm SL_2(\R)})$ contains non-zero elements, see \cite{DW}.
Thus Question \ref{Q1} (1.a) is  reduced to the non-geometric
3-manifolds containing hyperbolic JSJ pieces (the so-called mixed 3-manifold). 
In view of Theorem
\ref{basic property of volume of presentation} (2) (3), as well as
the result of \cite{DW}, and in an attempt to seal a relation
between the \emph{Gromov simplicial volume} and the \emph{hyperbolic
volume}, M.~Boileau and several others wondered the
following more direct version of Question \ref{Q1} (1):

\begin{question}\label{Q2}
	Suppose that the Gromov norm $||M||$ is positive.
	\begin{enumerate}
		\item Is there a representation $\rho\co\pi_1  M\to{\rm PSL}(2;{\C})$
		with positive volume?
		\item More weakly is there a representation $\rho\co\pi_1 \t M\to{\rm
		PSL}(2;{\C})$ with positive volume for some finite covering $\t M$
		of $M$?
	\end{enumerate}
\end{question}

From now on $M$ will always be assumed to be a closed oriented irreducible
non-geometric 3-manifold.

The main results of this paper are the following two theorems,
which answer Questions \ref{Q1} and \ref{Q2} respectively.

\begin{theorem}\label{virt-non-zero}
	Suppose that $M$  is a closed oriented irreducible non-geometric 3-manifold.
	\begin{enumerate}
		\item If $M$ contains at least one hyperbolic geometric piece,
        then the hyperbolic volume of some finite cover $\tilde{M}$ is positive.
        \item If $M$ contains at least one Seifert geometric piece,
        then the Seifert volume of some finite cover $\tilde{M}$ is positive.
	\end{enumerate}
\end{theorem}

\begin{theorem}\label{zero}
	\ 	
	\begin{enumerate}
		\item There are closed oriented non-trivial graph manifolds
		with vanishing Seifert volume.
		\item There are closed oriented irreducible 3-manifolds
		with non-vanishing Gromov norm but vanishing hyperbolic volume.
	\end{enumerate}
\end{theorem}

\begin{corollary}\label{non-cover}
Neither the hyperbolic volume nor the Seifert volume have the covering property.
 \end{corollary}

Now let us have a brief discussion of our proofs of the main results.

The difficulty of Question \ref{Q1} can more or less be seen from the
definition: to get a non-zero element in ${\rm vol}\left(M, G
\right)$ we need first to find an \emph{a priori} ``significant''
representation $\rho\co\pi_1M\to G$, and then to be able to compute
its volume. 
In the geometric case, there
is a natural significant representation given by the faithful and
discrete representation of its fundamental group in the Lie group of
its geometry.    
In the non-geometric case one might think to use the geometry
of its  pieces  to construct a
global significant representation. However in this new situation
many problems occur: First the geometric pieces have non-empty
boundary and the volume of representation is not easy to manipulate
 and moreover we must make sure that the local
representations are compatible in the toral boundaries in order to
be glued together.    
Then another problem arises when we want to
compute the volume of a global representation from the local
volumes.  Can we add the volumes of those pieces to get the volume of the presentation?

In order to prove Theorem \ref{virt-non-zero} and Theorem \ref{zero}, we will first consider the volume
of representations from the perspective of Chern--Simons theory and  prove the so-called additivity principle.

Denote by $G$ the semi-simple Lie group ${\rm Iso}_e\t{{\rm
SL}_2(\R)}$ or ${\rm PSL}(2,{\C})$ with the associated Riemannian
homogeneous spaces $X$ which  is $\t{{\rm SL}_2(\R)}$ or ${\Hi}^3$
endowed with the closed $G$-invariant volume form $\omega_X$.

Denote by $\mathfrak{g}$ the Lie algebra of $G$. We recall (see Section 3
for more details) that the Chern--Simons classes with structure group
${\rm PSL}(2,{\C})$ are based on the first Pontrjagin class and in
the same way we define the Chern--Simons classes with structure group
${\rm Iso}_e\t{{\rm SL}_2(\R)}$  based on the invariant polynomial
defined by ${\bf R}(A\otimes A)={\rm Tr}(X^2)+t^2$ where $A$ is an
element of the Lie algebra of ${\rm Iso}_e\t{{\rm SL}_2(\R)}$ which
decomposes into $X+t$ where $X$ is in the Lie algebra of $\t{{\rm
SL}_2(\R)}$ and $t\in{\R}$. Denote
the imaginary part of the complex number $z$ by $\Im(z)$.

\begin{proposition}\label{vol} 
	Let $\rho$ be a  representation of $\pi_1M$ into $G$
	and $A$ be a  corresponding flat $G$-connection in the
	principal bundle $P=M\times_\rho G$. Suppose that $P$ admits a section
	$\delta$ over $M$.
	
	\begin{enumerate}
		\item If $G$ equals ${\rm Iso}_e\t{{\rm SL}_2(\R)}$ then 
		\begin{equation}
			\mathfrak{cs}_M(A,\delta)=\int_M\delta^{\ast}{\bf R}\left(\ud A\wedge
			A+\frac{1}{3}A\wedge[A,A]\right)=\frac{2}{3}{\rm vol}_G(M,\rho).
		\end{equation}
		In particular, the Chern--Simons invariant of flat ${\rm Iso}_e\t{{\rm
		SL}_2(\R)}$-connections is gauge invariant.
		\item If $G$ equals ${\rm PSL}(2;{\C})$ then
		\begin{equation}
			\Im\left(\mathfrak{cs}_M(A,\delta)\right)=-\frac{1}{\pi^2}{\rm
			vol}_{G}(M,\rho).
		\end{equation}
	\end{enumerate}
\end{proposition}

\begin{remark} 
Assuming that $P=M\times_\rho G$ admits a section in Proposition \ref{vol} (1) means equivalently
that $\rho$ admits a lift into $\t{{\rm SL}_2(\R)}$ so that the
bundle admits a reduction to an $\t{{\rm SL}_2(\R)}$-bundle and we
reckon that the correspondence  for
$G=\t{{\rm SL}_2(\R)}$  is pointed in \cite{Re-rationality}, and verified in
\cite{Kh} by  a long and subtle computation. However for our own
understanding we reprove it in a very simple way  underscoring that
the correspondence is quite natural and comes directly from the
structural equations of the Lie group involved (Section 3.3).
The correspondence in Proposition \ref{vol}  (2) is derived from \cite{KK}, using a formula established by Yoshida in \cite{Yo}  (Section 3.4).
The imaginary part of the Chern--Simons invariants of flat ${\rm
PSL}(2;{\C})$-connections is gauge invariant from the formula for it
does not depend on the chosen section. 
\end{remark}

Then from the representations in normal form developed in \cite{KK}, we have the so-called 
additivity principle, that is,  we can add the volumes of those pieces to get the globe volumes.
We state it in the following  

\begin{theorem}\label{additivity}
            Let $M$ be an irreducible oriented closed $3$-manifold with JSJ tori $T_1,\cdots,T_r$ and JSJ pieces 
            $J_1,\cdots,J_k$,
            and let $\zeta_1,\cdots,\zeta_r$ be slopes on $T_1,\cdots T_r$, respectively.
           
            Suppose that $G$ is either ${\rm Iso}_e\t{{\rm SL}_2(\R)}$ or ${\rm PSL}(2;{\C})$, and that 
                $$\rho:\pi_1(M)\to G$$
            is a representation vanishing on the slopes $\zeta_i$,  and that 
                $\hat\rho_i:\pi_1(\hat{J_i})\to G$
            are the induced representations, where $\hat{J}_i$ is the Dehn filling of $J_i$ along
            slopes adjacent to its boundary, with the induced orientations. Then

$${\rm
vol}_G(M,\rho)={\rm vol}_G({\hat J}_1,\hat{\rho}_1)+ {\rm
vol}_G({\hat J}_2,\hat{\rho}_2)+\ldots+ {\rm
vol}_G({\hat J}_k,\hat{\rho}_k).$$ 
\end{theorem}

With Proposition \ref{vol} and Theorem \ref{additivity}  at hands, for  a given 3-manifold $M$ satisfying the conditions of Theorem \ref{virt-non-zero},  how do we contstruct a finite cover $\tilde M$ of $M$  and a representation  $\rho : \pi_1(\tilde M)\to G$ with positive volume? 
Such a cover $\tilde M$ and a representation $\rho$ are not difficult to describe up to some finite cover and conjugation as we see below.

The prime picture of the cover $\tilde M$ is simple. Fix a JSJ piece $J$ of $M$ with the right geometry. 
The regular finite cover $\tilde{M}$ of $M$ can be cut along some of its JSJ tori into three parts: 
Part 1 are some disjoint preimage components (elevations) of $J$, Part 2 some ``corridors'',
and Part 3 the remaining part;
components in Part 1 and Part 3 are connected by those corridors in Part 2. 
The crucial property of corridors is that in each corridor $X$ there is a corridor surface $R$,
so that if $X$ is a corridor connecting some $\tilde J$ in Part 1 through a component $\tilde T$
of $\partial \tilde J$, then $\partial R$ has exactly one component (circular doorsill) in $\tilde T$. 

It seemed hopeless to find such corridors in general, until the recent striking results of surfaces in 3-manifolds 
due to Wise and his co-authors \cite{HW,Wise-long,PW-graph,PW-mixed,PW3}.
In this paper we will construct such corridors (Theorem \ref{corridorCover}) from what we call
parallel-cutting partial Przytycki--Wise subsurfaces (Theorem \ref{virtualPartialPW}), 
then we will merge the three parts above to provide the designed finite cover $\tilde M$ 
(Proposition \ref{mergeFiniteCovers}, Corollary \ref{virtualExtensionOfSemicovers}).
To understand the key issue that we will address,
consider the following situation: If $S$ is a $\pi_1$-injectively immersed, virtually embedded,
connected, closed subsurface of $M$, and if $T$ is a JSJ torus of $M$, is there a regular finite
cover $\tilde{M}$ of $M$ such that any elevation of $S$ intersects any elevation of $T$ in at most
one connected component? Generally speaking, the answer seems to be negative when $S$ and $T$ intersects
in more than one component. In particular, the double separability
between $S$ and $T$ does not automatically lead to such a cover $\tilde{M}$. 
However, the main new input of Theorem \ref{corridorCover} is 
the parallel cutting condition, which morally assumes that $S$ intersects every JSJ torus of $M$ in virtually
parallel components (Definition \ref{parallelCutting}). Under this and other 
suitable working assumptions,
we will be able to resolve the key issue by applying the separability criterion of Rubinstein--Wang \cite{RW}
in graph submanifolds of $M$ and the relatively quasiconvex separability of Wise \cite{Wise-long} in hyperbolic pieces.

Thus the term ``surface separability'' in the title of this paper
mostly refers to virtual simplification of intersection
between an immersed subsurface and any JSJ torus, rather than intersection of an immersed subsurface with itself.

We describe the representation $\rho :\pi_1(\tilde M)\to G$ (Theorem \ref{virtualExtensionRep}) as follows: 
$\rho|$ on the group of each component $\tilde J$ of Part 1  
factors through a ``mighty and natural'' representation of the $\pi_1(\hat J)$, where $\hat J$ is  the closed 3-manifold obtained by 
Dehn filling of $\tilde J$ along its doorsills and the representation is discrete and faithful (the corridor surfaces can be chosen to be so);
$\rho|$ on  the group of each component  of Part 3 is trivial; and $\rho|$ on  the group of each corridor is
based on the crucial property of its corridor surface and   is given by the homological pairing
$\pi_1(X)\longrightarrow H_1(X;\,\ZZ)  \stackrel{[R]}{\longrightarrow} \ZZ{\longrightarrow} G,$ 
where everything is oriented, $[R]\in H_2(X,\partial X;\ZZ) \cong H^1(X;\ZZ)$.  
The global representation $\rho$ can be obtained by gluing together the local representations $\rho|$  provided they induce conjugate representations on each boundary components of the Parts 1, 2 and 3 (Lemma \ref{cut-paste} or Remark \ref{also-cut-paste}).
To make sure that the local representations are conjugate in the boundary, we need the concepts of
 colored chunks and colored merging, where  colored chunks can be merged with matching color in a further finite cover 
(Lemma \ref{coloredMerging})  and  also 
 certain so called "class invertible properties" of the groups ${\rm Iso}_e\t{\rm SL_2(\R)}$, 
and ${\rm PSL}(2;\C)$ (Lemma \ref{geometricClassInv}). Finally  the conjugacy issue above will be managed in Lemma \ref{glueRep}

To prove Theorem \ref{zero}, some arguments in \cite{EHN}, an example of Motegi \cite{Mo} and a result of Hoffman--Matignon \cite{HM}
are also used.

The organization of this paper is reflected by the table of contents.
Efforts have been made in organizing the materials 
so that our results can be verified smoothly, and readers
can access the topic more easily. 

\bigskip\noindent\textbf{Acknowledgement}. 
We thank Professor Michel Boileau and Professor Daniel Matignon for
helpful communications.

\section{Preliminaries}\label{Sec-preliminaries}
	In this section, we review the geometric decomposition of $3$-manifolds and volume of representations
	of closed manifolds.

	\subsection{Topology of 3-manifolds after Thurston}
		Let $N$ be a connected compact prime orientable 3-manifold with toral or empty boundary. 


		\subsubsection{Geometric decomposition}
		As consequence of the geometrization of $3$-manifolds \cite{Th1,Th2} 
		achieved by G.~Perelman and W.~Thurston,
		exactly one of the following holds:
		\begin{itemize}
			\item Either $N$ is geometric, supporting one of the following eight geometries: ${\Hi}^3$,
			$\widetilde{{\rm SL}_2({\R})}$, ${\Hi}^2\times{\R}$, ${\rm Sol}$, ${\rm Nil}$, ${\R}^3$, ${\S}^3$ and
			${\S}^2\times {\R}$ (where ${\Hi}^n$, ${\R}^n$ and ${\S}^n$ are the $n$-dimensional
			hyperbolic space, Euclidean space, and spherical space respectively);
			\item or $N$ has a canonical nontrivial geometric decomposition.
			In other words, there is a nonempty minimal union $\c{T}_N\subset N$ of disjoint 
			essential tori and Klein bottles of $N$, unique up to isotopy, such
			that each component of $N\setminus\c{T}_N$ is either Seifert fibered or atoroidal.
			In the Seifert fibered case the piece supports  the ${\Hi}^2\times{\R}$ geometry and 
			the $\widetilde{{\rm SL}_2({\R})}$ geometry,
			 and
			in the atoroidal case the piece supports the ${\Hi}^3$ geometry. 
		\end{itemize}
		When $N$ has nontrivial geometric decomposition, we call the components of $N\setminus\c{T}_N$ 
		the \emph{geometric pieces} of $N$, or more specifically,
		\emph{Seifert pieces} or \emph{hyperbolic pieces} according to their geometry.
		We call $N$ a \emph{mixed 3-manifold} if $N$ contains at least one hyperbolic piece, 
		or a \emph{graph manifold} otherwise.
		
		{\it A chunk} of $N$ is a submanifold that is a union of a subset of geometric pieces, 
		glued up along the cut tori between them.
		A graph chunk is a chunk which is a graph manifold.
		
		Traditionally, there is another decomposition introduced
		by Jaco--Shalen \cite{JS} and Johannson \cite{Joh}, known as the \emph{JSJ decomposition}.
		When $N$ contains no essential Klein bottles and has a nontrivial geometric decomposition,
		the JSJ decomposition of $N$ coincides with its geometric decomposition,
		so the cut tori and the geometric pieces may be referred to as the \emph{JSJ tori} and 
		the \emph{JSJ pieces}, respectively.
		Possibly after passing to a double cover of $N$, we may assume that $N$ contains
		no essential Klein bottle. In fact, 
		the following lemma of virtual reduction is well known, (cf.~\cite[Lemma 3.1]{PW3}).
		
		\begin{lemma}\label{not-shared}
			If $N$ is a closed prime 3-manifold which is not geometric,
			then there exists a regular finite covering $\tilde N$ of $N$ satisfying the following:
			\begin{itemize}
				\item $\tilde N$ contains no essential Klein bottle;
				\item each JSJ torus of $\tilde N$ is adjacent to a pair of distinct JSJ pieces;
				\item each Seifert piece is a product of $S^1$ and a compact surface of positive genus.
			\end{itemize}
		\end{lemma}
		
		\subsubsection{Hyperbolic pieces}\label{bh}
		A hyperbolic piece $J$ can be realized
		as a complete hyperbolic $3$-manifold of finite volume,
		unique up to isometry by Mostow Rigidity. 
		With respect to the geometry, any properly $\pi_1$-injectively
		immersed connected subsurface $j:\,S\looparrowright J$ of finite type,
		is either \emph{geometrically finite} or \emph{geometrically infinite},
		unless $\pi_1(S)$ is elementary (trivial or infinite cyclic). 
		Geometrically infinite subsubsurfaces are virtual fibers;
		geometrically finite subsurfaces 
		are quasi-Fuchsian so any conjugate of $\pi_1(S)$ in $\pi_1(J)$ is
		quasiconvex relatively to the cusps. It follows from the work
		of Wise \cite{Wise-long} that $\pi_1(S)$ is separable in $\pi_1(J)$
		in the cusped case, and the closed case follows from
		Agol's proof of the Virtual Haken Conjecture \cite{AGM}.
		In particular, $S$ is virtually embedded in $J$. 
		
		We briefly recall the Thurston's Hyperbolic Dehn Filling Theorem.
		Let $J$ denote a compact, orientable $3$-manifold whose boundary
		consists of tori $T_1,\ldots,T_p$ and whose interior admits a complete
		hyperbolic metric. We denote by $J_{\mathrm{max}}$ the interior of $M$ with a
		system of maximal cusps removed. Identify $J$ with
		$J_{\mathrm{max}}$, then $\partial J$ has a Euclidean metric induced
		from the hyperbolic metric and each closed Euclidean geodesic in
		$\partial J$ has the induced length. 
		The Hyperbolic Dehn Filling Theorem \cite[Theorem 5.8.2]{Th1}
		can be stated in the following form.
		
		\begin{theorem}\label{surgered}
			Let $J$ be a compact oriented $3$-manifold with toral boundary
			$T_1\cup\ldots\cup T_p$ whose interior admits a complete
			hyperbolic structure. Identify $J$ with $J_{max}$.
			Then there is a positive constant $C$ such that the
 			the closed 3-manifold  $J(\zeta_1,\ldots,\zeta_n)$ obtained by 
 			Dehn filling each $T_i$ along $\zeta_i$ admits a complete hyperbolic
			structure if each slope $\zeta_i$ has then length greater than $C$.
			Moreover, with suitably chosen base points, $J(\zeta_1,\ldots,\zeta_n)$
			converges to $J$ in the Gromov--Hausdorff sense as the minimal length
			of $\zeta_i$ tends to infinity.
		\end{theorem}

 		\subsubsection{Seifert pieces}
 		A Seifert piece $J$ of a non-geometric prime closed $3$-manifold $N$ supports 
 		both the ${\Hi}^2\times{\R}$ geometry and the $\widetilde{{\rm SL}_2({\R})}$ geometry.
 		In this paper we are more interested in the latter, so we describe the structure
 		of $\widetilde{{\rm SL}_2(\R)}$ geometric manifolds in the following.
 				
We consider the group ${\rm PSL}(2;{\R})$  as the orientation
preserving isometries of the hyperbolic
  $2$-space ${\Hi}^2=\{z\in{\C}, \Im(z)>0\}$ with $i$ as a base point. In this way ${\rm PSL}(2;{\R})$ is a (topologically trivial) circle bundle over ${\Hi}^2$.
  Denote by $p\co\t{{\rm SL}_2(\R)}\to{\rm PSL}(2;{\R})$ the universal covering of ${\rm PSL}(2;{\R})$ with the induced metric. Then $\t{{\rm SL}_2(\R)}$
  is a  line bundle over ${\Hi}^2$. For any $\alpha\in{\R}$, denote by ${\rm sh}(\alpha)$ the element of $\t{{\rm SL}_2(\R)}$
  whose projection into ${\rm PSL}(2;{\R})$ is given by $\begin{pmatrix}
\cos(2\pi\alpha)&\sin(2\pi\alpha) \\
-\sin(2\pi\alpha)&\cos(2\pi\alpha)
\end{pmatrix}$. Then the set $\{{\rm sh}(n), n\in{\Z}\}$,
  is the kernel of $p$ as well as the center of $\t{{\rm SL}_2(\R)}$, acting by integral translation along the fibers of $\t{{\rm SL}_2(\R)}$.
  By extending this ${\Z}$-action on the fibers by the ${\R}$-action we get the whole identity component of the isometry group of $\t{{\rm SL}_2(\R)}$.
  To summarize we have the following diagram of central extensions
$$\xymatrix{
\{0\} \ar[r] \ar[d] & {\Z} \ar[r] \ar[d] & \t{{\rm SL}_2(\R)}\ar[r] \ar[d] & {\rm PSL}(2;{\R})\ar[r] \ar[d] & \{1\} \ar[d] \\
\{0\} \ar[r]  & {\R} \ar[r]  & {\rm Iso}_e\t{{\rm SL}_2(\R)}
\ar[r]  & {\rm PSL}(2;{\R})\ar[r]  & \{1\} }.$$ In particular the
group ${\rm Iso}_e\t{{\rm SL}_2(\R)}$ is generated by $\t{{\rm
SL}_2(\R)}$ and the image of ${\R}$ which intersect together in the
image of ${\Z}$, where each element $x$ on ${\R}$ is naturally identified with the
translation $\tau_x$ of length $x$. More precisely we state the following useful lemma which is easy to check.

\begin{lemma}\label{SL}
We have the identification ${\rm Iso}_e\t{{\rm SL}_2(\R)}={\R}\times_{\Z}\t{{\rm SL}_2(\R)}$:
where 
$(x,h)\sim({x'},h')$ if and only if there exists an integer
$n\in{\Z}$ such that ${x'}-x=n$ and $h'={\rm sh}(-n)\circ h$. 
\end{lemma}

 Let  $F_{g,n}$ be an
oriented  $n$-punctured surface of genus $g\geq 0$ with boundary
components $s_1,\ldots,s_n$ with $n\ge 0$. Then
$J'=F_{g,n}\times{\S}^1$ is oriented if ${\S}^1$ is oriented. Let
$h_i$ be the oriented ${\S}^1$-fiber on the torus $T_i=s_i\times h_i$. We say that $(s_i,h_i)$ is a \emph{section--fiber}  basis of $T_i$. Let $0\leq s\leq n$. Now attach $s$ solid tori $V_i$'s to the
boundary tori $T_i$'s of $J'$ such that the meridian of $V_i$ is identified
with the slope ${a_i}s_i+{b_i}h_i$ where $a_i>0,
(a_i,b_i)=1$ for $i=1,\ldots,s$. Denote the resulting manifold
by $\left(
g,n-s;\frac{b_1}{a_1},\cdots,\frac{b_s}{a_s}\right)$
which has the Seifert fiber structure extended from the circle
bundle structure of $J'$. Each orientable Seifert fibered space with
orientable base $F_{g,n-s}$ and with $\leq s$ exceptional fibers is
obtained in such a way.
If $J$ is closed, i.e. if $s=n$, then  define the Euler number of
the Seifert fibration by
$$e(J)=\sum_{i=1}^s \frac{b_i}{a_i}\in\mathbb{Q}$$
and  the Euler characteristic of the orbifold $O(J)$ by
$$\chi_{O(J)}=2-2g-\sum_{i=1}^s\left(1-\frac 1{a_i}\right)\in\mathbb{Q}.$$

From \cite{BG2} we know that a closed orientable $3$-manifold $J$ supports the $\widetilde {{\rm
SL}_2(\R)}$ geometry, i.e.~there is a discrete and faithful representation
$\psi: \pi_1J\to {\rm Iso} \t{{\rm SL}_2(\R)}$, if and only if
$J$ is Seifert fibered with non-zero Euler number $e(J)$ and
negative Euler characteristic $\chi_{O(J)}$ of the base orbifold.

A properly $\pi_1$-injectively immersed subsurface $j:S\looparrowright J$ of finite type
is said to be \emph{horizontal} if it can be properly homotoped to be transverse to the fiber at any point.
Otherwise it is said to be \emph{vertical}, and in this case, it is an annulus or a torus fibering over
a properly immersed path or loop in the base orbifold.

\subsection{Volume of representations of closed manifolds}
We recall three definitions of volume of representations.

\subsubsection{Via developing maps} 
Given a  semi-simple, connected Lie  group $G$ and a closed oriented
manifold $M^n$ of the same dimension as  the contractible space
$X^n=G/K$, where $K$ is a maximal compact subgroup of $G$, we can
associate, to each representation $\rho\co\pi_1M\to G$, a volume
${\rm vol}_G(M,\rho)$ in the following way.

First  fix a
$G$-invariant Riemannian metric $g_X$ on $X$, and denote by
$\omega_X$ the corresponding $G$-invariant volume form. Let $\t{M}$ denote the universal covering of $M$. We think of the elements $\t{x}$ of  $\t{M}$
as the homotopy classes of paths $\gamma\co[0,1]\to M$ with $\gamma(0)=x_{0}$ which are acted by $\pi_1(M,x_0)$ by setting $[\sigma].\t{x}=[\sigma.\gamma]$, where $.$ denotes the paths composition.

A developing map $D_{\rho}\co\t{M}\to X$ associated to $\rho$ is a
$\pi_1M$-equivariant map such that
for any $x\in \t{M}$ and $\alpha \in \pi_1M$, then
$$D_{\rho}(\alpha.x)=\rho(\alpha)^{-1}D_{\rho}(x)$$ where $\rho(\alpha)$ acts on $X$ as an isometry.
Such a map does exist and can be constructed explicitly as in
\cite{BCG}: Fix a triangulation $\Delta_M$ of $M$. Then its lift is a
triangulation $\Delta_{\t{M}}$ of $\t{M}$, which is $\pi_1M$-equivariant.
Then fix a fundamental domain $\Omega$ of $M$ in $\t{M}$ such that
the zero skeleton $\Delta^0_{\t{M}}$ misses the frontier of $\Omega$. Let
$\{x_1,\ldots,x_l\}$ be  the vertices of $\Delta^0_{\t{M}}$ in $\Omega$, and
let $\{y_1,\ldots,y_l\}$ be  any $l$ points in $X$.  We first set
$$D_{\rho}(x_i)=y_i, \,\, i=1 ,\ldots, l.$$
Next extend $D_{\rho}$ in an $\pi_1M$-equivariant way to $\Delta^0_{\t{M}}$:
For any vertex $x$ in $\Delta^0_{\t{M}}$, there is a unique  vertex
$x_i$ in $\Omega$ and $\alpha_x\in \pi_1M$ such that
$\alpha_x.x_i=x$, and we set
$D_{\rho}(x)=\rho(\alpha_x)^{-1}D_{\rho}(x_i)$. Finally we extend
$D_{\rho}$ to edges, faces, etc., and $n$-simplices of $\Delta_{\t{M}}$
by straightening the images to geodesics using the homogeneous
metric on the contractible space $X$.
This  map is unique up to equivariant homotopy. Then
$D_{\rho}^{\ast}(\omega_X)$ is a $\pi_1M$-invariant closed $n$-form
on $\t{M}$ and therefore can be thought of as a closed $n$-form on $M$. Thus define
$${\rm vol}_G(M,\rho)=\int_MD_{\rho}^{\ast}(\omega_X)=\sum_{i=1}^s\epsilon_i {\rm vol}_X(D_{\rho}(\t\Delta_i))$$
where $\{\Delta_1,\ldots,\Delta_s\}$ are the $n$-simplices of $\Delta_M$,
$\t \Delta_i$ is a lift of $\Delta_i$ and $\epsilon_i=\pm 1$ depending on whether $D_{\rho}|\t\Delta_i$ is preserving  or reversing orientation.

\subsubsection{Via continuous cohomology classes}\label{contcoh} 
Let $\mathfrak{g}$ and $\mathfrak{k}$ denote the Lie algebra of $G$ and $K$. Let
$o=\{K\}$ be the base point of $X=G/K$ and for any $g_1,\ldots,g_l\in G$
denote by $\Delta(g_1,\ldots,g_l)$ the geodesic $l$-simplex of $X$ with
vertices $\{o,g_1(o),\ldots,g_l\ldots g_2g_1(o)\}$.   There is a natural
homomorphism
$$H^{\ast}(\mathfrak{g},\mathfrak{k};{\R})=H^{\ast}(G{\rm -invariant\ differential\ forms\ on\ }X)\to H^{\ast}_{\rm cont}(G;{\R})$$
defined in \cite{Du-simplicial}  by $\eta\mapsto
\left((g_1,\ldots,g_l)\to\int_{\Delta(g_1,\ldots,g_l)}\eta\right)$ which
turns out to be an isomorphism by the Van Est Theorem \cite{V}.

Recall that for each representation $\rho\co\pi_1M\to G$ one can
associate a flat bundle over
 $M$ with fiber $X$ and group $G$ constructed as follows: $\pi_1M$ acts diagonally on the product
 $\t{M}\times X$  by the following formula
 \begin{eqnarray}
\sigma.(\t{x},g)=(\sigma.\t{x},\rho^{-1}(\sigma)g)
 \end{eqnarray}
 and we can form the quotient $\t{M}\times_\rho X=(\t{M}\times X)/\pi_1M$ which is the flat $X$-bundle over
 $M$ corresponding to $\rho$.

Then  for each $G$-invariant closed form $\omega$ on $X$,
$q^*(\omega)$ is a $\pi_1(M)$-invariant closed form on $\t{M}\times
X$, where $q: \t{ M}\times X\to X$ is the projection, which induces
a form $\omega'$ on $M\times _\rho X$. Then $s^*(\omega')$ is a
closed form on $M$, where $s: M\to M\times _\rho X$ is a section (since $X$ is
contractible, such a section exists  all such  sections are
homotopic).
Thus any representation  $\rho\co\pi_1M\to G$ leads to a natural homomorphism
$$\rho^*: H^{\ast}_{\rm cont}(G;{\R})=H^{\ast}(G{\rm -invariant\ differential\ forms\ on\ }X)\to  H^{\ast}(M;{\R})$$
induced by $\rho^*(\omega)=s^*{\omega}'$. The volume of $\rho$ is therefore defined by 
$${\rm vol}_G(M,\rho)=\int_M\rho^{\ast}(\omega_X)$$

  The equivalence between the two definitions is immediate since the
$\pi_1M$-equivariant map ${\rm Id\times D_\rho}\co\t{M}\to \t
M\times X$ descends to a section $M\to M\times _\rho X$.

\subsubsection{Via transversely projective foliations}\label{foliation}
This definition only makes sense for the Seifert volume.
Let $\mathfrak{F}$ be a co-dimension one foliation on a closed
smooth manifold $M$ determined by a 1-form $\omega$. Then by the
Froebenius Theorem one has $d\omega= \omega\wedge\delta$ for some
1-form $\delta$. It was observed by Godbillon and Vey \cite{GV} that
the 3-form $\delta \wedge
 d\delta$ is closed and the class $[\delta \wedge
 d\delta]\in H^3(M,{\R})$ depends only on the foliation $\mathfrak{F}$ (and not on the chosen form $\omega$). This cohomology class is termed
\emph{the Godbillon--Vey class} of the foliation  $\mathfrak{F}$ and denoted by
 $GV(\mathfrak{F})$.

 \begin{proposition}[{\cite[Proposition 1]{BG1}}]\label{Euler}  Suppose $\mathfrak{F}$
  is a horizontal flat foliation  on a circle  bundle ${\S}^1\to E \to M$ with structural group ${\rm PSL}_2(\R)$. Then
 $$\int_{\S^1} GV(\mathfrak{F})=4\pi^2\t{e}(E),$$
 where $\int_{\S^1}\co H^3(E)\to H^2(M)$ denotes the integration along the fiber and $\t{e}$ denotes the Euler class of the bundle.
 \end{proposition}

Let $M$ be a closed orientable 3-manifold and $\phi\co\pi_1M\to {\rm
PSL}_2(\R)$ be a representation with zero Euler class. Since ${\rm
PSL}_2(\R)$ acts on ${\S}^1$ then one can consider the corresponding
flat circle bundle $M\times_{\phi}{\S}^1$ over $M$ and the
associated horizontal $({\rm PSL}_2(\R),{\S}^1)$-foliation
$\mathfrak{F}_{\phi}$. Since the Euler class of $\phi$ is zero we
can choose a section $\delta $ of $M\times_{\phi}{\S}^1\to M$.
Brooks and Goldman showed that
$\delta^{\ast}GV(\mathfrak{F}_{\phi})$ only depends on $\phi$ (and
not on a chosen section $\delta$) \cite[Lemma 2]{BG1}. Then they defined  the \emph{Godbillon--Vey} invariant of
$\phi$ by setting
$$GV(\phi)=\int_M\delta^{\ast}GV(\mathfrak{F}_{\phi}).$$


For a given representation $\phi\co\pi_1M \to{\rm PSL}_2(\R)$, $\phi$
lifts to $\t\phi\co\pi_1M \to{\t{{\rm SL}_2(\R)}}$ if and only if
$\t{e}(\phi)=0$ in $H^2(M,{\Z})$.
The following fact has been verified in \cite{BG1}.

\begin{proposition}\label{CV=Vol} Let $M$ be a closed oriented
$3$-manifold, let $\phi\co\pi_1M\to {\rm PSL}_2(\R)$ be a representation
with zero Euler class  and fix a lift  $\t\phi\co\pi_1M \to
 \t{{\rm SL}_2(\R)}$ of $\phi$. Then
$$GV(\phi)={\rm vol}_{\t{{\rm SL}_2(\R)}}(M,\t\phi),$$
where $\t{{\rm SL}_2(\R)}$ is viewed as a semi-simple Lie group acting on itself by multiplication with corresponding homogeneous space $\t{{\rm SL}_2(\R)}$.
\end{proposition}

\section{Chern--Simons theory and additivity principle}

In this section, we consider volume of representations from the perspective of Chern--Simons theory.
In particular, we prove the additivity principle (Theorem \ref{additive}).

Throughout this section we refer to
\cite{CS} and \cite{KN}. In this part, all the objects we deal with
are smooth. Let $\pi\co P\to M$ denote a principal $G$-bundle over a
closed manifold $M$. We suppose that $G$ is a Lie group acting on the right on $P$ and we denote by $R_g$ the right action
$$P\ni x\mapsto x.g\in P$$ where $g$ in an element of $G$.
 Denote by $\mathfrak{g}$ the Lie algebra of $G$. Let $VP$ be the vertical subbundle of $TP$.

 Let $P_1$ and $P_2$ denote two principal  $G_1$, respectively, $G_2$-bundles  over manifolds $M_1$ respectively,  $M_2$. Following the formalism in \cite[p.~53]{KN}   a \emph{homomorphism of principal bundles} consists of a map  $f\co P_1\to P_2$ as well as a homomorphism $f'\co G_1\to G_2$ such that $f(x.g)=f(x).f'(g)$, where $x\in P_1$ and $g\in G_1$. We say that a bundle homomorphism induces the \emph{identity in the structural group} if $G_1=G_2=G$ and $f'$ is the identity map.

\subsection{Connections on principal bundles}
  We denote by $\Omega^k(P; \mathfrak{g})$ the set of differential $k$-forms taking values in $\mathfrak{g}$.  We define the exterior product of $\omega_1\in\Omega^k(P; \mathfrak{g})$ by   $\omega_2\in\Omega^l(P; \mathfrak{g})$ as an element $\omega_1\wedge\omega_2$ of $\Omega^{k+l}(P; \mathfrak{g}\otimes\mathfrak{g})$ by setting
$$\omega_1\wedge\omega_2(X_1,\ldots,X_{k+l})=$$ $$\frac{1}{(k+l)!}\sum_{\sigma\in\mathfrak{S}_{k+l}}{\rm sign}(\sigma)\omega_1(X_{\sigma(1)},\ldots,X_{\sigma(k)})\otimes\omega_2(X_{\sigma(k+1)},\ldots,X_{\sigma(l)}).$$

The Lie bracket $[ ., . ]$ in $\mathfrak{g}$ induces a map $\Omega^{k+l}(P; \mathfrak{g}\otimes\mathfrak{g})\to\Omega^{k+l}(P; \mathfrak{g})$ and we denote $[\omega_1,\omega_2]$ the image of $\omega_1\wedge\omega_2$ under this maps. Explicitly we get:
$$[\omega_1,\omega_2](X_1,\ldots,X_{k+l})=$$ $$\frac{1}{(k+l)!}\sum_{\sigma\in\mathfrak{S}_{k+l}}{\rm sign}(\sigma)[\omega_1(X_{\sigma(1)},\ldots,X_{\sigma(k)}),\omega_2(X_{\sigma(k+1)},\ldots,X_{\sigma(l)})].$$
The differential $d\co\Omega^k(P; \mathfrak{g})\to\Omega^{k+1}(P; \mathfrak{g})$ is defined by the Cartan formula
$$d\omega(X_1,\ldots,X_{k+1})=\frac{1}{k+1}\sum_{i=1}^{k+1}X_i.\omega(X_1,\ldots,\hat{X_i},\ldots,X_{k+1})+$$ $$\frac{1}{k+1}\sum_{i<j}(-1)^{i+j}\omega([X_i,X_j],X_1,\ldots,\hat{X_i},\ldots,\hat{X_j},\ldots,X_{p+1}).$$

The derivative at the identity $1$ of $G$ of the map $$G\ni g\mapsto x.g\in P$$ induces an isomorphism $\nu_x\co\mathfrak{g}\to V_xP\subset T_xP$ and we get the exact sequence
$$0\to\mathfrak{g}\stackrel{\nu_x}{\to} T_xP\stackrel{d\pi_x}{\to}T_{\pi(x)}M\to 0.$$
A horizontal subbundle $HP$ of $TP$ is a smooth distribution such that $T_xP=V_xP\oplus H_xP$ for any $x\in P$ that is $G$ equivariant: $H_{x.g}=dR_g(x)H_x$. This is given equivalently by the kernel  of an element $\omega\in\Omega^1(P;\mathfrak{g})$ such that for any $x\in P$

(1) $\omega_x\circ\nu_x={\rm Id}_{\mathfrak{g}}$ and

(2) $R_g^{\ast}\omega={\rm Ad}_{g^{-1}}(\omega)$.

An element of $\Omega^1(P;\mathfrak{g})$ satisfying (1) and (2) is
termed  \emph{a  connection of} $P$. Denote by $\c{A}(P) $ the space
of all conections on $P$. This space is naturally acted on by the
\emph{gauge group} denoted by  $\c{G}_P$  consisting of the
$G$-equivariant  bundle automorphisms of $P$.

The basic example is the group $G$ itself, viewed as a trivial
bundle over a point or more generally the trivialized bundle
$M\times G$ with the so-called Maurer--Cartan connection
$\omega_{{\rm M.C.}}=d(L_{g^{-1}}\circ\pi_2)$, where $L_g$ denotes
the left translation in $G$ and $\pi_2$  the projection of $P$ onto
$G$. This connection satisfies the Maurer--Cartan equation, namely
$$d\omega_{{\rm M.C.}}=-\frac{1}{2}[\omega_{{\rm M.C.}},\omega_{{\rm
M.C.}}].$$

Let us make a concrete computation for $G$. Let
$X_1,\ldots,X_n$ be a basis of $\mathfrak{g}$.   Since $\mathfrak{g}$ can be thought of as the space of left invariant vector fields in $G$, its dual $\mathfrak{g}^*$ is the space of left invariant differential 1-forms on $G$. Let
$\theta^1,\ldots,\theta^n$ denote the dual basis of $\mathfrak{g}^*$. Then
$$\omega_{{\rm M.C.}}=\theta^1\otimes X_1+\ldots+\theta^n\otimes X_n.$$ Let us write
the constants structure of $\mathfrak{g}$ which are given by the
formula $$[X_j,X_k]=\sum_ic^i_{jk}X_i.$$ Thus by  the Maurer--Cartan
equation we get the equalities
\begin{eqnarray}
d\theta^i=-\frac{1}{2}\sum_{j,k}c^i_{j,k}\theta^j\wedge\theta^k
\end{eqnarray}
In general,  for a given connection $\omega$ in a bundle $P$, the element
 \begin{eqnarray}
 F^{\omega}=d\omega+\frac{1}{2}[\omega,\omega]
 \end{eqnarray}
 is the \emph{curvature} of $\omega$ lying in $\Omega^2(P;\mathfrak{g})$ and measuring the integrability of the corresponding horizontal distribution.  When $F^{\omega}=0$ we say that the connection is flat. Denote by  $\c{FA}(P)$ the subset of $\c{A}(P)$ which consists of flat connections on $P$. This space is preserved by the gauge group action.

We recall the following basic fact that will be used very often in
this paper. To each flat connection $\omega$ one can associate a
representation $\rho\co\pi_1M\to G$ by lifting the loops of $M$ in
the leaves of the horizontal foliation given by integrating the
distribution $\ker\omega$.


On the other hand $\omega$ can be recovered from $\rho$ by the
following construction. The fundamental group of $M$ acts on the
product $\t{M}\times G$ by the formula
$[\sigma].([\gamma],g)=([\sigma.\gamma],\rho([\sigma]^{-1}).g)$ and
the quotient $\t{M}\times_{\rho} G$ under this $\pi_1M$-action is
isomorphic to $P$ and the push forward of the vertical distribution
of $\t{M}\times G$ in $\t{M}\times_{\rho} G$ corresponds to $\omega$
in $P$. We get a natural map
$$I_P\co\c{B}(P)=\c{FA}(P)/\c{G}_P\hookrightarrow\c{R}(\pi_1M,G)/{\rm conjugation},$$
where $\c{R}(\pi_1M,G)$ is the set of representations of $\pi_1M$
into $G$ acted by the conjugation in $G$. 

\subsection{Chern--Simons classes}
Given a  Lie group $G$, a polynomial of degree $l$ is a symmetric
linear map $f\co\otimes^l\mathfrak{g}\to {\mathbb K}$, where
${\mathbb K}$ denotes either the real or  the complex numbers field.
The group $G$ acts on $\mathfrak{g}$ by ${\rm Ad}$ and the
polynomials invariant under this action are called the
\emph{invariant polynomials of degree} $l$ and are denoted by
$I^l(G)$ with the convention $I^0(G)={\mathbb K}$. Denote $I(G)$ the
sum $\oplus_{l\in{\N}} I^l(G)$.

The \emph{Chern--Weil theory} gives a correspondence $W_P$ from
$I^l(G)$ to $H^{2l}(M;{\mathbb K})$ constructed in the following
way. Choose a connection $\omega$   in $P$ then for any $l\geq 1$ a
polynomial $f\in I^l(G)$ gives rise to a $2l$-form
${f}(\wedge^lF^{\omega})$ in $P$. It follows from the Chern--Weil
Theory that ${f}(\wedge^lF^{\omega})$ is closed and  is the
pull-back of a unique form on $M$ under $\pi\co P\to M$ denoted by
$\pi_{\ast}{f}(\wedge^lF^{\omega})$. Then $W_P(f)$ is by definition
the class of $\pi_{\ast}{f}(\wedge^lF^{\omega})$ in $H^{2l}(M)$. The
Chern--Weil Theorem claims that $W_P(f)$ does not depend on the
chosen connection $\omega$ and that $W_P$ is actually a
homomorphism.

Let $EG$ denote the \emph{universal principal} $G$-bundle and denote by $BG$ the classifying space of $G$. This means that any principal $G$-bundle $P\to M$ admits a bundle homomorphism $\xi\co P\to EG$ descending to the classifying map, still denoted $\xi\co M\to BG$, that is unique up to homotopy. There exists the universal Chern--Weil homomorphism $\t{W}\co I^l(G)\to H^{2l}(BG)$ such that $\xi^{\ast}\t{W}(f)=W_P(f)$.

The Chern--Simons invariants were derived from this construction by Chern and Simons who
observed that ${f}(\wedge^lF^{\omega})$, for $l\geq 1$, is actually exact in $P$
and a primitive is given explicitly in \cite{CS} by
\begin{eqnarray}
Tf(\omega)=l\int_0^1f(\omega\wedge(\wedge^{l-1}F^t))dt
\end{eqnarray}
where $F^t=tF^{\omega}+\frac{1}{2}(t^2-t)[\omega,\omega]$. The form
$Tf(\omega)$ is closed when $M$ is of dimension $2l-1$. For instance
when $l=2$ and $M$ is a $3$-manifold, plugging $F^t$ and (3.2) into
(3.3) we get a closed $3$-form on $P$, namely
\begin{eqnarray}
Tf(\omega)=f(F^{\omega}\wedge\omega)-\frac{1}{6}f(\omega\wedge[\omega,\omega])=f(d{\omega}\wedge\omega)+\frac{1}{3}f(\omega\wedge[\omega,\omega])
\end{eqnarray}
Considering $G$ as a principal bundle over the point this yields to $$Tf(\omega_{{\rm M.C.}})=-\frac{1}{6}f(\omega_{{\rm M.C.}}\wedge[\omega_{{\rm M.C.}},\omega_{{\rm M.C.}}]).$$ The $(2l-1)$-form $Tf(\omega_{{\rm M.C.}})$ is closed, bi-invariant and defines a class in $H^{2l-1}(G;{\R})$.
Let us denote by
$$I_0(G)=\{f\in I(G), Tf(\omega_{{\rm M.C.}})\in H^{2l-1}(G;{\Z})\}.$$
 The elements of $I_0(G)$ are termed \emph{integral polynomials}. If $f\in I_0(G)$ then there is a well defined functional

\begin{eqnarray}
 \mathfrak{cs}^*_{M}\co\c{A}_{M\times G}\to{\mathbb K}/{\Z}
 \end{eqnarray}
 defined as follows: since $P=M\times G$  is a trivial(ized)  we can consider, for any section $\delta$, the Chern--Simons invariant
 \begin{eqnarray}
 \mathfrak{cs}_M(\omega,\delta)=\int_M\delta^{\ast}Tf(\omega)
\end{eqnarray}
Since $f$ is an integral polynomial, the element $\mathfrak{cs}_M(\omega,\delta)$ is well defined modulo ${\Z}$ when the section changes. Then define $\mathfrak{cs}^{\ast}_M(\omega)$ to be   the class of $\mathfrak{cs}_M(\omega,\delta)$ in ${\mathbb K}/{\Z}$.

The fundamental classical examples  are $G={\rm SU}(2;{\C})$ and
$G={\rm SO}(3;{\R})$.

The Chern--Simons classes for the group ${\rm SU}(2;{\C})$ are based on the second Chern class $f=C_2\in I_0^2({\rm SL}(2;{\C}))$.
We recall that the Chern classes,  denoted by $C_1$, $C_2$ for ${\rm SU}(2;{\C})$, are the complex valued invariant polynomials
such that $${\det}\left(\lambda. I_2-\frac{1}{2i\pi}A\right)=\lambda^2+C_1(A)\lambda+C_2(A\otimes A),$$ when $A\in\mathfrak{sl}_2(\C)$. Thus after developing this equality we get
$$C_2(A\otimes A)=\frac{1}{8\pi^2}{\rm tr}(A^2),$$ so that we get the usual
formula (using (3.4))
\begin{eqnarray}
TC_2(\omega)&=\frac{1}{8\pi^2}{\rm
Tr}\left(F^{\omega}\wedge\omega-\frac{1}{6}\omega\wedge[\omega,\omega]\right)
\end{eqnarray}
\begin{eqnarray*}
=\frac{1}{8\pi^2}{\rm
Tr}\left(d{\omega}\wedge\omega+\frac{1}{3}\omega\wedge[\omega,\omega]\right)
\end{eqnarray*}

The Chern-Simons classes of the special orthogonal group $G={\rm
SO}(3;{\R})$ are based on the first Pontrjagin class $f=P_1\in
I_0^2({\rm SO}(3;{\R}))$ that is a   the real valued invariant
polynomial such that $${\det}\left(\lambda.I_3-\frac{1}{2\pi}A\right)=\lambda^3+P_1(A\otimes A)\lambda,$$
when $A\in\mathfrak{so}_3(\R)$.
Thus after developing this equality we get
$$P_1(A\otimes A)=-\frac{1}{8\pi^2}{\rm tr}(A^2).$$

\begin{example}
  When $M$ is an oriented Riemaniann closed $n$-manifold one can consider its associated ${\rm SO}(n;{\R})$-bundle ${\rm SO}(M)$  which consists of the  positive orthonormal unit frames endowed with the \emph{Levi Civita connection}.   When $M$ is of dimension $3$ it is well known that its is parallelizable so that there exist sections $\delta$ of  ${\rm SO}(M)\to M$. Therefore one can consider the Chern-Simons invariant of the Levi Civita connection on $M$ that will be denoted by $\mathfrak{cs}_{\rm L.C.}(M,\delta)$.
\end{example}

A natural question arises in the following situation. There is an
epimorphism $\pi_2\co{\rm SU}(2;{\C})\to{\rm SO}(3;{\R})$ that is
the $2$-fold universal covering. Thus any connection $\omega$ on the
trivialized ${\rm SU}(2;{\C})$-bundle over $M$ induces a connection
$\omega'$ on the corresponding   $ {\rm SO}(3;{\R})$-bundle over
$M$. How can we compute $TP_1(\omega')$ from $TC_2(\omega)$? The
answer is given in \cite[pp 543, end of Section 3]{KK} by recalling
that $\pi_2$ induces a homomorphism between the corresponding
classifying spaces  $$\pi_2^{\ast}\co H^4(B{\rm SO}(3;{\R}))\to
H^4(B{\rm SU}(2;{\C})),$$ such that
$$\pi_2^{\ast}\t{W}(P_1)=-4\t{W}(C_2).$$ Thus using the definition
and the Chern--Weil universal  homomorphism we get the equality
\begin{eqnarray}
\mathfrak{cs}_{M}(\omega',\delta')=-4\mathfrak{cs}_{M}(\omega,\delta)
\end{eqnarray}

where $\delta$ is a fixed section in the ${\rm SU}(2;{\C})$-bundle
over $M$ and $\delta'$ is the corresponding section in the $ {\rm
SO}(3;{\R})$-bundle over $M$. On the other hand since $G={\rm
SO}(3;{\R})$, respectively, ${\rm SU}(2;{\C})$, are the maximal compact
subgroup of ${\rm PSL}(2;{\C})$, respectively, ${\rm SL}(2;{\C})$, whose
quotients ${\rm PSL}(2;{\C})/{\rm SO}(3;{\R})$, respectively, ${\rm
SL}(2;{\C})/{\rm SU}(2;{\C})$ are contractible then it follows from
\cite[Chapter 15, Theorem 3.1]{Ho} and 
\cite[Proposition 7.2, p.~98]{Du-curvature} 
that the natural inclusion gives rise to isomorphisms
$H^{\ast}(B {\rm PSL}(2;{\C}))\to H^{\ast}(B {\rm SO}(3;{\R}))$ and
$H^{\ast}(B {\rm SL}(2;{\C}))\to H^{\ast}(B {\rm SU}(2;{\C}))$. We
have the following commutative diagram
$$\xymatrix{
H^*(B{\rm PSL}(2;{\C})) \ar[r]^{\simeq} \ar[d] & H^*(B {\rm SO}(3;{\R})) \ar[d]\\
H^*(B{\rm SL}(2;{\C})) \ar[r]^{\simeq} & H^*(B {\rm SU}(2;{\C})) }.$$
Hence we also get (fixing a trivialization, using (3.6), (3.7),
(3.8))
\begin{eqnarray}
 \mathfrak{cs}_{M}(\omega',\delta')&=-4\mathfrak{cs}_{M}(\omega,\delta)
 \end{eqnarray}
\begin{eqnarray*} =
-\frac{1}{2\pi^2}\int_M\delta^{\ast}{\rm
Tr}\left(F^{\omega}\wedge\omega-\frac{1}{6}\omega\wedge[\omega,\omega]\right)
\end{eqnarray*}
 where $\delta$ is a fixed section in the ${\rm
SL}(2;{\C})$-bundle over $M$ and $\delta'$ is the corresponding
section in the $ {\rm PSL}(2;{\C})$-bundle over $M$.

\subsection{Volume and Chern--Simons classes in Seifert geometry}
In this section we check Proposition \ref{vol} (1) keeping the same notation as in the introduction.
The  proof is inspired from \cite[p.~532]{BG2} and we  we will
follow faithfully their presentation. If $G={\rm Iso}_e(\t{{\rm
SL}_2({\R})})$ then the matrices 
$$X=\begin{pmatrix}
1&0 \\
0&-1
\end{pmatrix}\textrm{, }Y=\begin{pmatrix}
0&0 \\
1&0
\end{pmatrix}\textrm{, and }Z=\begin{pmatrix}
0&1 \\
0&0
\end{pmatrix},$$
together with the generator $T$ of   ${\R}$ form a basis of the Lie algebra $\mathfrak{g}$ of $G$.
Setting $W=Z-Y-T$ we get a new basis $\{X,Y,Z,W\}$ of $\mathfrak{g}$
with commutators relations
\begin{eqnarray}
[X,Y]=-2Y, [X,Z]=2Z,
\end{eqnarray}
\begin{eqnarray*}
[Y,Z]=[Y,W]=[Z,W]=-X, [X,W]=2Y+2Z
\end{eqnarray*}
 which determine the
coefficients in the Maurer--Cartan equations. Denote by
$\varphi_X,\varphi_Y,\varphi_Z,\varphi_W$ the dual basis of
$\mathfrak{g}^{\ast}$. The Maurer--Cartan form of $G$ is given by
$$\omega_{\rm M.C.}=\varphi_X\otimes X+\varphi_Y\otimes
Y+\varphi_Z\otimes Z+\varphi_W\otimes W.$$ Denote by $A$ a flat
connection on $M\times_{\rho}G$. By Section 3.1,  if $\t{M}$ denotes
the universal covering and if  $q\co \t{M}\times G\to G$ denotes the
projection, then  $A$ corresponds to the form
$\o{q^{\ast}(\omega_{\rm M.C.})}$, where $-\co\t{M}\times
G\to\t{M}\times_{\rho} G$ denotes the push-forward which makes sense
since $q^{\ast}(\omega_{\rm M.C.})$ is $\pi_1M$-invariant.   The
Chern--Simons class of the flat connection  $A$ is $T{\bf
R}(A)=\o{q^{\ast}T{\bf R}(\omega_{\rm M.C.})}$. Using  equations
(3.1) and (3.10), we calculate
\begin{eqnarray}
d\varphi_X= & \varphi_Y\wedge\varphi_Z+\varphi_Y\wedge\varphi_W+\varphi_Z\wedge\varphi_W\\
d\varphi_Y= & 2\varphi_X\wedge\varphi_Y-2\varphi_X\wedge\varphi_W\\
d\varphi_Z= & -2\varphi_X\wedge\varphi_Z-2\varphi_X\wedge\varphi_W\\
d\varphi_W= & 0
\end{eqnarray}
Notice that those equations also imply that
$2(\varphi_X\wedge\varphi_Y+\varphi_X\wedge\varphi_Z)=d(\varphi_Y-\varphi_Z)$
and therefore
$$T{\bf R}(\omega_{\rm M.C.})=\frac{2}{3}\varphi_X\wedge\varphi_Y\wedge\varphi_Z+\frac{1}{3}  d(\varphi_Y\wedge\varphi_W-\varphi_Z\wedge\varphi_W).$$
The end of the proof follows from the  commutativity of the diagram
below and from the Stokes formula, since
$\varphi_X\wedge\varphi_Y\wedge\varphi_Z$ represents the volume form
on $X=\t{{\rm SL}_2({\R})}$.
$$\xymatrix{
G \ar[r] & X \\
 \t{M}\times  G\ar[u]^{q_G} \ar[d]_{-} \ar[r]^{\t{\pi}} & \t{M}\times X \ar[u]_{q_X}  \ar[d]^{-}  \\
  M\times_\rho G \ar[r]^{\pi}  & M\times_\rho X \\
M \ar[u]^{\delta} \ar[ur]_{s} & }$$
This completes the proof of
Proposition \ref{vol} (1).

\subsection{Volume  and Chern--Simons classes in hyperbolic geometry}
We now check Proposition \ref{vol} (2).  The following construction is largely inspired from  \cite[pp.~553--556]{KK}, using a formula established by Yoshida in \cite{Yo}.

Denote by $p\co{\rm PSL}(2;{\C})\simeq{\rm Iso}_+{\Hi}^3\to{\Hi}^3$ the natural projection.
For short denote ${\rm PSL}(2;{\C})$ by $G$.
For each representation $\rho\co\pi_1M\to G$ admitting a lift into ${\rm SL}(2;{\C})$, we have the (trivial)
principal bundle $M\times _\rho G$ and the associated bundle
$M\times _\rho {\Hi}^3$. Denote by $A$ the flat connection over $M$
corresponding to $\rho$ and $\omega_{{\Hi}^3}$ the $G$-invariant volume form on
${\Hi}^3$ corresponding to  the hyperbolic metric.

The matrices $X=\begin{pmatrix}
1&0 \\
0&-1
\end{pmatrix}$, $Y=\begin{pmatrix}
0&0 \\
1&0
\end{pmatrix}$, $Z=\begin{pmatrix}
0&1 \\
0&0
\end{pmatrix}$ form a basis of the Lie algebra $\mathfrak{sl}(2;{\C})$ with commutators relations $$[X,Y]=-2Y\textrm{, }[X,Z]=2Z\textrm{, }[Y,Z]=-X.$$ Denote by $\varphi_X,\varphi_Y,\varphi_Z$ the dual basis of $\mathfrak{sl}^{\ast}(2;{\C})$. The Maurer--Cartan form of $G$ is  
$$\omega_{\rm M.C.}=\varphi_X\otimes X+\varphi_Y\otimes Y+\varphi_Z\otimes Z,$$
and
$$TP_1(\omega_{\rm M.C.})=\frac{1}{\pi^2}\varphi_X\wedge\varphi_Y\wedge\varphi_Z.$$
By the formula of Yoshida in \cite{Yo} we know that
$$iTP_1(\omega_{\rm M.C.})=\frac{1}{\pi^2}p^*\omega_{{\Hi}^3}+i\mathfrak{cs}_{\rm L.C.}({\Hi}^3)+d\gamma,$$
where $p^*\omega_{{\Hi}^3}$ is the pull-back of $\omega_{{\Hi}^3}$
under the projection $p: {\rm PSL}(2;{\C})\to{\Hi}^3$,
$\mathfrak{cs}_{\rm L.C.}({\Hi}^3)$ is the Chern--Simons 3-form of
the Levi Civita connection over ${\Hi}^3$ (see Example 3.1) with the
hyperbolic metric in its ${\rm SO}(3)$-frame  bundle ${\rm
PSL}(2;{\C})$   and $d\gamma$ is an exact real form. Consider
the following commutative diagram
$$\xymatrix{
G \ar[r]^{{p}} & {\Hi}^3 \\
 \t{M}\times  G\ar[u]^{q_G} \ar[d]_{-} \ar[r]^{{p}} & \t{M}\times {\Hi}^3 \ar[u]_{q_{{\Hi}^3}}  \ar[d]^{-}  \\
  M\times_\rho G \ar[r]^{p}  & M\times_\rho {\Hi}^3 \\
M \ar[u]^{\delta} \ar[ur]_{s} & }$$
Notice that the sections in the bottom triangle are obtained as follows. Since $M$ is a $3$-manifold then it follows from the obstruction theory that any principal bundle with simply connected group is trivial. Since $\rho\co\pi_1M\to G$ admits a lift into ${\rm SL}(2;{\C})$ $M\times_\rho G$ is trivial.  Denote by $\delta$ a section of  $M\times_\rho G\to M$. It induces, by $p\circ\delta=s$, a section of  $M\times_\rho {\Hi}^3\to M$.

Since all the maps are clear from the context, in the sequel, we will drop the index in the projections $q_G$ and $q_{{\Hi}^3}$ and we denote them just by $q$. Now the 3-form $\omega_{{\Hi}^3}$ induces a 3-form
$\o{q^*\omega_{{\Hi}^3}}$ on $M\times_\rho {\Hi}^3$ and
$$i\o{q^*TP_1(\omega_{\rm M.C.})}=\frac{1}{\pi^2}\o{q^*p^*\omega_{{\Hi}^3}}+i\o{q^*\mathfrak{cs}_{\rm L.C.}({\Hi}^3)}+\o{q^*d\gamma}$$
in $M\times_\rho G$, where the push-forward operation $\o{q^{\ast}(.)}$  indeed makes sense since $TP_1(\omega_{\rm M.C.}), p^*\omega_{{\Hi}^3}$ and $\mathfrak{cs}_{\rm L.C.}({\Hi}^3)$ are left invariant forms in $G$. Then 
$$i\mathfrak{cs}_M(A,\delta)=\frac{1}{\pi^2}\int_M\delta^*\o{q^*p^*\omega_{{\Hi}^3}}+
i\int_M\delta^*\o{q^*\mathfrak{cs}_{\rm L.C.}({\Hi}^3)}+\int_M\delta^*\o{q^*d\gamma}.$$
Since
$\delta^*\o{q^*p^*\omega_{{\Hi}^3}}=\delta^*p^*\o{q^*\omega_{{\Hi}^3}}=s^*\o{q^*\omega_{{\Hi}^3}}$
and $\int_M\delta^*\o{q^*d\gamma}=0$ by the Stokes Formula,  we have
$$i\mathfrak{cs}_M(A,\delta)=\frac{1}{\pi^2}\int_Ms^*\o{q^*\omega_{{\Hi}^3}}+
i\int_M\delta^*\o{q^*\mathfrak{cs}_{\rm L.C.}({\Hi}^3)}=\frac{1}{\pi^2}{\rm vol}(M,\rho)+i\mathfrak{cs}(M_{\rho};\delta),$$
where we denote $\int_M\delta^*\o{q^*\mathfrak{cs}_{\rm L.C.}({\Hi}^3)}$ by
 $\mathfrak{cs}(M_{\rho};\delta)$.
 We get eventually
 \begin{eqnarray}
\mathfrak{cs}_M(A,\delta)=\mathfrak{cs}(M_{\rho};\delta)-\frac{i}{\pi^2}{\rm vol}(M,\rho)
\end{eqnarray}

This completes the proof of
Proposition \ref{vol} (2).

\subsection{Normal form near toral boundary of  3-manifolds}
In this part we recall the machinery developed in \cite{KK}.
Let $M$ be a compact oriented $3$-manifold with toral
boundary $\b M$ endowed with a  preferred basis $s,h$ of $H_1(\b M;{\Z})$ (this implies that for each component $T_i$ of $\partial M$, there is a basis  $s_i,h_i$, but for simplicity, we omit the sub-index). 
Let $\rho\co\pi_1M\to G$ be a representation where $G$ is either
${\rm PSL}(2;{\C})$ or $\t{{\rm
SL}(2;{\R})}$. We consider the space of flat connections $\c{FA}({P})$ where $P$ is the trivialized bundle $M\times G$. For representations into  $\t{{\rm
SL}(2;{\R})}$ the corresponding principal bundles  are always trivial whereas the representations $\rho$ into  ${\rm PSL}(2;{\C})$ leading to a trivial bundle are precisely those who admit a lift $\o{\rho}$ into ${\rm SL}(2;{\C})$. Moreover if follows from \cite{KK} and \cite{Kh} that after a conjugation, the representation $\rho|\pi_1T$ can be put in \emph{normal form}, which either \emph{hyperbolic}, \emph{elliptic} or \emph{parabolic}. Since the parabolic form will not be used in the explicit way we only recall  the definitions of those representations which are
\emph{elliptic} or \emph{hyperbolic} in the  ${\rm PSL}(2;{\C})$ case and \emph{elliptic} in the $\t{{\rm
SL}(2;{\R})}$ case in the boundary of $M$.   Then  by \cite{KK}, when $G={\rm PSL}(2;{\C})$, we
may assume, after conjugation, that there exist $\alpha,\beta\in{\C}$ such that
$$\rho(s)=\begin{pmatrix}
e^{2i\pi\alpha}&0 \\
0&e^{-2i\pi\alpha}
\end{pmatrix}\  {\rm and}\ \rho(h)=\begin{pmatrix}
e^{2i\pi\beta}&0 \\
0&e^{-2i\pi\beta}
\end{pmatrix};$$
when $G=\t{{\rm
SL}(2;{\R})}$, after conjugation, we may assume that  after  projecting to ${\rm PSL}(2;{\R})$ there exist $\alpha,\beta\in{\R}$ such that
$$s\mapsto\begin{pmatrix}
\cos(2\pi\alpha)&\sin(2\pi\alpha) \\
-\sin(2\pi\alpha)&\cos(2\pi\alpha)
\end{pmatrix}\ \textrm{ and}\ h\mapsto\begin{pmatrix}
\cos(2\pi\beta)&\sin(2\pi\beta) \\
-\sin(2\pi\beta)&\cos(2\pi\beta)
\end{pmatrix}.$$ In either case, if  $A$ denotes a connection on $P$ corresponding to $\rho$ then after a gauge transformation $g$ the connection $g*A$ is in normal form: $$g*A|T\times[0,1]=(i\alpha dx+i\beta dy)\otimes X.$$


Let $M$ be a closed oriented  3-manifold and $\mathcal T$ be a union of finitely many tori  cutting  $M$ into 
$ J_1,\ldots, J_k$. For  each $T$ in $\c{T}$, we  endow  a  homology basis $(m_T,l_T)$ and  a base point $x_T=m_T\cap l_T$; 
and for simplicity we assume $T$ shared by $J_j$ and $J_l$, $j\ne l$ (this condotion can be reached in a finite cover, see Lemma 2.1). With the setting above, we  have the following cut and paste result according to  the correspondence between connections in normal form and representations for manifolds with toral boundary due to Kirk and Klassen (the similar fact has been used in \cite{DW}).

\begin{lemma}\label{cut-paste}
   Let $\rho_i\co\pi_1J_j \to G$ be a elliptic/hyperbolic representation, where $G$ is either $\t{{\rm SL}(2;{\R})}$ or  ${\rm PSL}(2;{\C})$,
  $i=1,...,k$. If for each $T$ the induced representations $\rho_j|\pi_1T $ and $\rho_l|\pi_1 T$ are conjugated in $G$, 
then there exists a global representations $\rho\co\pi_1M\to G$ inducing $\rho_i$ over $J_i$ up to conjugacy, $i=1,...,k$.
\end{lemma}

\begin{proof} For each $T$ in $\c T$,  $\rho_i$ induces $\rho_i|_T:\pi_1(T\times[-1,1],x_T)\to G$, where $i=j,l$, $T\times[-1,1]$ is a regular neighborhood of $T$ with $T\times[-1,0]\subset J_j$ and $T\times[0,1]\subset J_l$.

Denote by $A_i$ be a flat connection over $J_i$ corresponding to $\rho_i$. 
Then ${A_i}|T\times[-1,1]$ can be put into normal form after  gauge-transformation. Specially, there exists $g_{i, T}\co T\times[-1,1]\to G$, $i=j,l$, such that 
$g_{i, T}{\ast}A_i|T\times[-1,1]=(i\alpha_{i, T} dx+i\beta_{i, T} dy)\otimes X$. 
By obstruction theory one can extend $$\coprod_{T\in\partial J_i}g_{i, T}\to G$$ to gauge transformations $g_i\co J_i\to G$. 

Since for each torus $T$,  $\rho_j|_T$ and $\rho_l|_T$ are conjugated, they have the same eigenvalues,  therefore the connections $g_j*A_j$ and $g_l*A_l$ match on $T\times[-1,1]$. So their union define a flat and smooth connection $C$ over $M$ and therefore a representation $\rho$ of $\pi_1M$ into 
$G$.
\end{proof}




We quote the  following result  stated in \cite[Lemma 3.3]{KK} with $G={\rm SL}(2;{\C})$ and in \cite[Theorem 4.2]{Kh} with $G=\t{{\rm SL}(2;{\R})}$, that will be used latter :
\begin{proposition}\label{KKh}
Let $A$ and $B$ denote two flat connections in normal form over an oriented $3$-manifold with toral boundary. If $A$ and $B$ are equal near the boundary and if they are gauge equivalent then

i)  $\mathfrak{cs}_M(A,\delta)=\mathfrak{cs}_M(B,\delta)$ when $G=\t{{\rm SL}(2;{\R})}$   and,

ii)  $\mathfrak{cs}_M(A,\delta)-\mathfrak{cs}_M(B,\delta)\in{\Z}$ when $G={\rm PSL}(2;{\C})$.
\end{proposition}
The second statement follows from  \cite[Lemma 3.3]{KK} using
identity (3.9).
\begin{remark}\label{auto} As a consequence of  Proposition \ref{KKh}, if   $A$ and $B$ are flat connections on a solid torus that are equal near the boundary then the associated representations are automatically conjugated so that the conclusion of the proposition applies.
\end{remark}

 \subsection{Additivity principle}

 Fix a closed oriented  3-manifold $M$ and denote by $[M]$ its
 orientation class. Let $\mathcal T$ be a union of finitely many tori  cutting  $M$ into
 $ J_1,\ldots, J_k$.  For each $T\in \mathcal T$, suppose $T$ is  shared by $J_i$ and $J_j$. 
 Denote by $[J_i, \partial J_i]$
 the induced orientations classes so that
  the induced orientations on $\partial J_i$ and $\partial J_j$
 are opposite on $T$, and 
 we have $$[M]=\sum_{i=1}^k[(J_i,\b J_i)].$$

Fix a regular neighborhood $W(T)=[0,1]\times T$ such that
$T=\{1/2\}\times T$, $J_i\cap W(T)=[0,1/2]\times T$ and  $J_j\cap
W(T)=[1/2,1]\times T$. Let $A$ denote a flat connection over $M$.
Applying the same arguments as in \cite{KK} we may assume that  $A|W(T)$
is in normal form. Then by linearity of the integration
$$\mathfrak{cs}_{M}(A)=\sum_{i=1}^k\mathfrak{cs}_{J_i}(A|J_i).$$
Denote by $V$ the solid torus  with meridian
$m$.
Denote by $c$ a slope in $T$ and for each $T\in\mathcal T$, we perform a Dehn
filling to $J_i$ identifying $c$ with $m$ and denote by
$\hat{M}_i=J_i\cup\left(\cup_{c\subset \partial J_i} V_c\right)$ the resulting closed oriented manifold.
Suppose each $A|M_i$
smoothly extend to  flat connections over $\hat{J}_i$ denoted by $\hat{A}_i$ for $i\in \{1,\ldots,k\}$. 
This is to say that for any representation $\rho$ corresponding to $A$ then $[c]\in\ker\rho$.  By the linearity we have

$$\mathfrak{cs}_{\hat{J}_i}(\hat{A}_i)=\mathfrak{cs}_{J_i}(A|J_i)+\sum_{c\subset \partial J_i}\mathfrak{cs}_{V_c}(\hat{A}_i|V_c)\qquad (i).$$

Since the extensions from $J_i$ and $J_j$ over their own  $V_c$, based on the
normal form on $[0,1]\times T$, are the same on the $T$ direction but
opposite on the $[0,1]$ direction, then using Proposition \ref{KKh} and Remark \ref{auto} we have
$$\mathfrak{cs}_{V_c}(\hat{A}_i|V_c)+\mathfrak{cs}_{V_c}(\hat{A}_j|V_c)=0\qquad (*).$$

Summing up $(i)$ from $1$ to $k$, then apply (*), we get 
$$\mathfrak{cs}_{M}(A)=\sum_{i=1}^k\mathfrak{cs}_{\hat J_i}(A|\hat J_i).$$
Then applying Proposition \ref{vol} (1.2) and  (1.3) in the introduction to the former
equality we get the so-called additivity principle: 

\begin{theorem}\label{additive}
            Let $M$ is an oriented closed $3$-manifold with JSJ tori $T_1,\cdots,T_r$ and JSJ pieces 
            $J_1,\cdots,J_k$,
            and let $\zeta_1,\cdots,\zeta_r$ be slopes on $T_1,\cdots, T_r$, respectively.
           
            Suppose that $G$ is either ${\rm Iso}_e\t{{\rm SL}_2(\R)}$ or ${\rm PSL}(2;{\C})$, and that
                $$\rho:\pi_1(M)\to G$$
            is a representation vanishing on the slopes $\zeta_i$,  and that
                $\hat\rho_i:\pi_1(\hat{J_i})\to G$
            are the induced representations, where $\hat{J}_i$ is the Dehn filling of $J_i$ along
            slopes adjacent to its boundary, with the induced orientations. Then:
			$${\rm vol}_G(M,\rho)={\rm vol}_G({\hat J}_1,\hat{\rho}_1)+ {\rm vol}_G({\hat J}_2,\hat{\rho}_2)+
			\ldots+ {\rm vol}_G({\hat J}_k,\hat{\rho}_k).$$ 
\end{theorem}

We end this section by a simple lemma which will be used later and which is based on computation already developed in \cite{KK}.

\begin{lemma}\label{Almost trivial representation}
Suppose that $G$ is either 
${\rm Iso}_e\t{{\rm SL}_2(\R)}$ or ${\rm PSL}(2;{\C})$ and that $M$ is a closed oriented 3-manifold. If $\rho:\pi_1M\to G$ has image either infinite cyclic or finite, then  ${\rm
vol}_G({ M}, {\rho})=0$.
\end{lemma}

\begin{proof}
Suppose first the image $\rho(\pi_1M)$ is a cyclic group generated by $g$. Since $G$ is path connected,
there is a path connecting the unit $e$ and $g$ which provides a path of representation $\rho_t:\pi_1M\to G$ 
such that 
$\rho_1=\rho$ and $\rho_0$ is the trivial representation. 

Consider the associated path of flat connections $A_t$. This path
defines a connection $\mathbb{A}$ on the product $M\times[0,1]$ that is no longer flat but whose curvature
$F^{\mathbb{A}}$ satisfies the equation $F^{\mathbb{A}}\wedge
F^{\mathbb{A}}=0$ (this latter point follows from the fact that
$F^{A_t}=0$ for any $t$). Hence it follows from the construction of
the Chern--Simons invariant combined with the Stokes Formula that
$\mathfrak{cs}_{M}(A_1)=\mathfrak{cs}_{M}(A_0)=0$.
Hence  ${\rm
vol}_G( M,\rho)=0$ by Propsotion  \ref{vol}.

Suppose then the image  $\rho(\pi_1M)$ is a finite group $\Gamma$. Let $p: \tilde M\to M$ be the finite cover corresponding 
to  the unit of $\Gamma$, then we have the induced trivial representation $\rho\circ p_*: \pi_1\tilde M \to G$. Clearly ${\rm
vol}_G(\tilde M,\rho\circ p_*)=0$. Then by  ${\rm
vol}_G(\tilde M,\rho\circ p_*)=|\Gamma|{\rm
vol}_G(M,\rho)$ and therefore  ${\rm
vol}_G( M,\rho)=0$. 
\end{proof}


\section{Przytycki--Wise subsurfaces and separability}\label{PW}

	Sections \ref{PW} and \ref{Sec-virtualExtensionOfRepresentations} are prepared for 
	the construction part of the proof of Theorem \ref{virt-non-zero}.

	This section is inspired by
	recent work of P.~Przytycki and D.~Wise on surface subgroups
    of mixed $3$-manifolds \cite{PW-graph,PW-mixed,PW3}.
    We first review the merging trick
    which will be used repeatedly in the constructions of finite covers.
    Then we introduce the partial PW subsurfaces and the parallel cutting condition.
    We show that parallel-cutting partial PW subsurfaces with virtually prescribed
    boundary exist under very general conditions (Theorem \ref{virtualPartialPW}), 
    and that any parallel-cutting partial PW subsurface can virtually
    be arranged in nice position with respect to the JSJ tori (Theorem \ref{corridorCover}).
    These results are interesting on their own right 
    from the perspective of geometric topology, and they
    should be extendable to certain more natural contexts. 
    Besides techniques from Przytycki--Wise,
    the proofs essentially employ results of Wise \cite{Wise-long} 
    and Rubinstein--Wang \cite{RW} as well.
    
    For notational convenience and to avoid repetition we always, from now, consider mixed $3$-manifolds that contain no essential Klein
    bottles so that the JSJ decomposition coincide with the geometric decomposition.
    This causes no loss of generality as we are interested in virtual properties (Lemma \ref{not-shared}).

 	\subsection{Merging finite covers}
        In the study of virtual properties of mixed $3$-manifolds, we often need to construct
        a finite cover of a $3$-manifold from given covers of geometric pieces. This is possible
        via a procedure called \emph{merging}. 

        \begin{definition}
            Let $M$ be a compact orientable irreducible $3$-manifold
            with (possibly empty) incompressible toral boundary.
            For a positive integer $m$,
            we say that a finite cover $\tilde{M}$ is \emph{JSJ $m$-characteristic},
            if every elevation $\tilde{T}$ of a JSJ torus or
            of a boundary torus $T\subset\partial{M}$
            is the $m$-characteristic cover of $T$,
            namely, which means that every slope of $\tilde{T}$
            covers a slope of $T$ with degree $m$.
        \end{definition}

        \begin{proposition}\label{mergeFiniteCovers}
            Let $M$ be a compact orientable irreducible $3$-manifold
            with (possibly empty) incompressible toral boundary.
            Suppose $J'_1,\cdots,J'_s$
            are finite covers of all the JSJ pieces $J_1,\cdots,J_s$
            of $M$, respectively. Then there is a positive integer $m_0$, satisfying
            the following. For any positive integral multiple $m$ of $m_0$, there is
            a regular 
            finite cover $\tilde{M}$ of $M$, which is JSJ $m$-characteristic,
            such that any elevation $\tilde{J}_i$ of a JSJ piece $J_i$
            is a cover of $J_i$ that factors through $J'_i$.
        \end{proposition}

        \begin{proof}
            First observe that  if $M'$ is a JSJ $m$-characteristic finite cover of $M$,
            then there is a further JSJ $m$-characteristic regular finite cover $\tilde{M}$
            of $M$ that factors through $M'$. To see this, we may choose base points and assume $M'$ corresponds
            to a finite-index subgroup $\pi'$ of the pointed fundamental group $\pi$
            of $M$. Then the cover $\tilde{M}$ corresponding to the normal core
            $\tilde{\pi}=\cap_{g\in\pi}(g^{-1}\pi'g)$ of $\pi'$
            is clearly a regular finite cover of $M$.
             To see that it is JSJ $m$-characteristic, note that for 
            any torus subgroup $\tilde{P}\leq\tilde\pi$ that
            represents a JSJ or boundary torus $\tilde{T}$ of $\tilde{M}$, we have  $\tilde P=P\cap\tilde\pi$ for some
            JSJ or boundary torus $T$ of $M$. Since $P\cap g^{-1}\pi'g$ for any ${g\in\pi}$
            is the $m$-characteristic subgroup of $P$, and the $m$-characteristic subgroup of $P$ is unique, denoted as $P_m$,
            then $$\tilde P=P\cap\tilde\pi=P\cap (\cap_{g\in\pi}(g^{-1}\pi'g))=\cap_{g\in\pi} (P\cap(g^{-1}\pi'g))=\cap_{g\in\pi}  P_m=
            P_m$$ 
            is the $m$-characteristic subgroup
            of $P$ as well. In other words,
            any JSJ or boundary torus $\tilde{T}$ of 
            $\tilde{M}$ is also an $m$-characteristic cover
            of a JSJ or boundary torus $T$ of $M$.
            
            Furthermore, we may reduce the proof to the case when $M$ is either hyperbolic or Seifert fibered.
            In fact, assuming that we have proved that case, then applying the lemma to
            each $J_i$ allows us to take the positive integer $m_0(J_i)$ corresponding to each $J_i$.
            Let $m_0$ be the least common multiple of all $m_0(J_i)$.
            Then for any multiple $m$ of
            $J_i$, there is a further regular finite cover $\tilde{J}_i$ of $J_i'$,
            which is JSJ $m$-characteristic over $J_i$.
            Now (see \cite{Lu} 384-385) let $d_i$ be the degree of $\tilde{J}_i$ over $J_i$, and let $D$ be the least common multiple
            of all $d_i$. We take $\frac{D}{d_i}$ copies of each $\tilde{J}_i$. For any $T$ in $\b J_i$, there
            will be exactly $\frac{D}{d_i}\cdot\frac{d_i}{m}=\frac{D}{m}$ elevations
            for each side of $T$, and they are all $m$-characteristic
            over $T$. Thus we may glue these copies naturally along boundary, and a connected component
            $\tilde{M}$ will be a finite cover of $M$ which is JSJ $m$-characteristic. The observation at the beginning of the proof
            allows us to pass to a further regular finite JSJ $m$-characteristic cover of $M$.

            It remains to prove the result when $M$ is either hyperbolic or Seifert fibered.
            If $M$ is hyperbolic, the conclusion is implied by the omnipotence
            for cusped hyperbolic $3$-manifolds due to Wise \cite[Corollary 16.15]{Wise-long} (cf.~
            \cite[Lemma 3.2]{PW3}).
            Indeed, in this case, a cover $M'$ is given and there is no JSJ torus
            inside $M$.  Denote the boundary tori of $\partial M$ by $\{T\}$ and $\partial M'$ by $\{T'\}$. 
            By \cite[Lemma 3.2]{PW3}, there are finite covers $\{T''\}$ of  $\{T'\}$ such that 
            for any further finite cover $\{\tilde{T}\}$ of $\{T''\}$ 
            there is a finite cover $\tilde{M}$ of $M'$ so that the restriction on the boundary is the cover $\{\tilde T\}=\partial {\tilde M}\to \{T'\}$.  
             Therefore, we may pick a positive integer $m_0$
            and  $\{\tilde{T}\}$ above so that the composition  $\{\tilde{T}\}\to \{T\}$ is  $m_0$-characteristic, 
            and then the composition $\tilde M\to M$ is a desired cover.
            If $M$ is Seifert fibered, the conclusion can be seen directly. Indeed, in this case,
            suppose $S^1\to M\to\mathcal{O}$ is the Seifert fibration over the base orbifold $\mathcal{O}$.
            Let $m_0$ be the maximal order of torsion elements of $H_1(\mathcal{O};\ZZ)$. Since $M$ contains no
            essential Klein bottles, $H_1(M;\ZZ)\cong H_1(S^1;\ZZ)\oplus H_1(\mathcal{O};\ZZ)$. Thus,
            the cover $\tilde{M}_0$ corresponding to the kernel of $\pi_1(M)\to H_1(M;\ZZ_{m_0})$ is
            a regular finite cover of $M$ that is $m_0$-characteristic on the boundary. For any positive
            multiple $m$ of $m_0$, we may take $\tilde{M}$ to be the cover corresponding to the kernel
            of $\pi_1(\tilde{M}_0)\to H_1(M;\ZZ_{m/m_0})$. Note that $\tilde{M}_0$ is homeomorphic to
            a product of $S^1$ with an orientable compact surface, $\tilde{M}$ is a regular finite cover of $M$
            which is $m$-characterisitc on the boundary. This means $m_0$ is as desired.
        \end{proof}

         \begin{remark}  
			In the proof of  Proposition \ref{mergeFiniteCovers}, if the cover $\tilde{M}$ is chosen to be 
			corresponding to the subgroup $\tilde{\pi}=\cap_{\alpha\in \mathrm{Aut}(\pi')}\alpha(\pi')$, where $\mathrm{Aut}(\pi')$
            is the automorphism group of $\pi'$,
            it will be a characteristic finite cover of $M$, and similarly we can verify 
            that this cover is JSJ $m$-characteristic.
        \end{remark}

        \begin{definition}[{Cf.~\cite[Definition 3.3]{PW3}}]\label{semicover}
            For a compact orientable irreducible $3$-manifold $M$
            with (possibly empty) incompressible toral boundary, a \emph{semicover}
            of $M$ with respect to the JSJ decomposition is a compact orientable irreducible
            $3$-manifold $N$ together with an immersion $\mu:N\to M$, so that restricted to
            each component of $\partial N$, $\mu$ covers
            a JSJ torus of $M$.
        \end{definition}

        \begin{corollary}[{Cf.~\cite[Proposition 3.4]{PW3}}]\label{virtualExtensionOfSemicovers}
            With the notations above, if $\mu:N\to M$ is a semicover, then there is a JSJ $m$-characteristic  regular finite cover
            $\tilde{M}$ of $M$ in which a finite cover $\tilde{N}$ of $N$ is embedded as a chunk
            or a regular neighborhood of a JSJ torus, and
            the semicover $\tilde{N}\to M$ induced from $\mu$ is isotopic to the composition
            $\tilde{N}\hookrightarrow\tilde{M}\to M$.
        \end{corollary}

        \begin{proof}
            This has essentially been proved in \cite[Proposition 3.4]{PW3}, and we derive it
            from Proposition \ref{mergeFiniteCovers}.
            For each JSJ piece $J_i$ of $M$, if $J_i$ is  covered by a JSJ piece of  $N$, 
            let $J_i'$ be a common finite cover of all JSJ pieces of      $N$ that
            isotopically cover $J$, otherwise let $J_i'=J_i$.
             By Proposition \ref{mergeFiniteCovers}, we can obtain a JSJ 
             $m$-characteristic  regular finite cover $p'':M''\to M$, such that any elevation ${J}_i''$ of a JSJ piece $J_i$
            factors through $J'_i$.
            Let $\mu'':N''\to M''$ be any elevation of $\mu$, where $N''$ covers $N$.
            Suppose $J_*''$ is an elevation of a JSJ-piece $J_*$ of $N$, then $J_*''$ covers a JSJ-piece $J_i''$ of $M''$ under $\mu''$,
             where $J_i''$ is an elevation of  a JSJ-piece $J_i$ of $M$. By our construction we have $p''_*(\pi_1(J_i''))\subset \mu(\pi_1(J_*))$
             and therefore $\mu_*''|:  \pi_1(J_*'')\to \pi_1(J_i'')$ is surjective and $\mu''|:  J_*''\to J_i''$ is a homeomorphism.
            That is $\mu''$ is an embedding restricted to each JSJ piece of $N''$.
            
            Notice that the virtual embedding property is preserved after passing to a finite covering.
            Now either $N''$ is isotopic to a regular neighborhood
            of a JSJ torus of $M''$, or every JSJ piece of $N''$
            is mapped to  a unique JSJ piece of $M''$ by a homeomorphism.
            In the latter case, it follows that the induced map on the
            dual graph of the JSJ decompositions $\Lambda(N'')\to\Lambda(M'')$ is a canonical
            combinatorial local isometry, which is $\pi_1$-injective. Because
            $\pi_1(\Lambda(N''))$ is a free group of finite rank, and hence is LERF, 
             there is a finite cover $\tilde\Lambda$ of $\Lambda(M'')$,
            in which an elevation of $\Lambda(N'')$ is embedded as a complete subgraph.
            Therefore, we have  a regular finite JSJ $1$-characteristic covering $\tilde M\to M''$ so that any elevation  
            $\tilde \mu:\tilde N\to \tilde M$ of $\mu''$ is an embedding. 
            Therefore we have a JSJ $m$-characteristic covering $\tilde M\to M$ so that any elevation  
            $\tilde \mu:\tilde N\to \tilde M$ of $\mu$ is an embedding. 
            As we discussed at the beginning proof of  Proposition \ref{mergeFiniteCovers},
            by passing to a further finite covering, we may assume that  
            $\tilde M\to M$  is regular. 
		\end{proof}

   \subsection{PW subsurfaces and partial PW subsurfaces}

        \begin{definition}\label{PWsubsurface}
            Let $M$ be a closed orientable irreducible
            mixed $3$-manifold containing no essential Klein bottles.
            A \emph{Przytycki--Wise subsurface} of $M$,
            or simply a \emph{PW subsurface},
            is an immersed closed orientable   surface
                $$j:\,S\looparrowright M$$
            in minimal general position with respect to the JSJ decomposition,
            satisfying the following:
            \begin{itemize}
            \item $j$ is $\pi_1$-injective;
            \item for each maximal graph-manifold chunk $Q$ of $M$,
            each component of $j^{-1}(Q)$ is virtually embedded in $Q$; and
            \item for each hyperbolic piece $J$ of $M$, each component
            of $j^{-1}(J)$ is geometrically finite in $J$.
            \end{itemize}
            We may also regard any JSJ torus as a basic PW subsurface.
        \end{definition}

        In our discussion, it will usually be convenient to regard the unpointed
        fundamental group $\pi_1(M)$ of a $3$-manifold $M$ as the group
        of deck transformations on its universal cover $\widehat{M}$. Then,
        by a \emph{PW surface subgroup}, we mean the stabilizer of an
        elevation in $\widehat{M}$ of a PW subsurface,
        which depends on the choice of the elevation.

        Recall that a subset $W$ of a group $G$ is said to be \emph{separable} if it
        is closed in the profinite topology. More precisely, this means that for any
        $h\in G$ not contained in $W$, there is a finite quotient $\bar{G}$ in which
        $\bar{h}\not\in\bar{W}$. 

      
 The following lemma is a consequence of [PW2, Strong separation property] and [HW, Corollary 9.20].

        \begin{lemma}\label{PWseparability}
            Let $M$ be a closed orientable irreducible mixed $3$-manifold $M$.
            Then every PW surface subgroup of $\pi_1(M)$ is separable.
        \end{lemma}

        We introduce the notion of partial PW subsurfaces.

        \begin{definition}
            Let $M$ be an orientable closed irreducible mixed $3$-manifold.
            A \emph{partial PW subsurface} is a triple $(S,R,j)$ satisfying the
            following:
            \begin{itemize}
                \item $j:S\looparrowright M$ is a PW subsurface of $M$;
                \item $R\subset S$ is a connected compact essential subsurface; and
                \item every component of $\partial R$ is immersed into a JSJ torus under $j$.
            \end{itemize}
            We often ambiguously say that $j:R\looparrowright M$ is a partial PW subsurface,
            with the triple $(S,R,j)$ implicitly assumed. The \emph{boundary}
            of the partial PW subsurface is the boundary of $R$.
        \end{definition}

        \begin{definition}
            Let $M$ be an orientable closed irreducible mixed $3$-manifold
            containing no essential Klein bottles.
            Let $J_0$ be a JSJ piece, and $T_0$ be a JSJ torus adjacent to $J_0$,
            and $\zeta_0$ be a slope on $T_0$.
            A partial PW subsurface $j:R\looparrowright M$
            is said to be \emph{virtually bounded by $\zeta_0$ outside $J_0$}, if
            the boundary $\partial R$ of $R$ is nonempty,
            covering $\zeta_0$ under $j$, and if
            the interior $\mathring{R}$ of $R$ misses $J_0$ under $j$.
            In this case, the \emph{carrier chunk} $X(R)\subset M$
            of $R$ is the unique minimal chunk that contains $R$,
            and the \emph{carrier boundary} of $X(R)$ is
            the component $T_0\subset\partial X(R)$.
        \end{definition}

        \begin{definition}\label{parallelCutting}
            We say that a partial PW subsurface $j:R\looparrowright M$ is
            \emph{parallel cutting} if for every JSJ torus
            $T\subset M$, all components of $j^{-1}(T)$ in $R$
            cover the same slope of $T$.
        \end{definition}

    \subsection{Virtual existence of partial PW subsurfaces}

        \begin{theorem}\label{virtualPartialPW}
            Let $M$ be an orientable closed irreducible mixed $3$-manifold 
            containing no essential Klein bottles. 
            Let $\zeta_0$ be a slope on
            a JSJ torus $T_0$ adjacent to a JSJ piece $J_0$.
            Then for some finite cover $\tilde{M}$ of $M$ together with
            an elevation $(\tilde{J}_0,\tilde{T}_0,\tilde\zeta_0)$ of the triple $(J_0,T_0,\zeta_0)$,
            there exists
            a parallel-cutting partial PW subsurface $\tilde{R}\looparrowright \tilde{M}$,
            bounded virtually by $\tilde\zeta_0$ outside $\tilde{J}_0$.
        \end{theorem}

        \begin{proof}
            We need to strengthen some arguments in
            the work of Przytycki--Wise \cite{PW-graph,PW-mixed}.
            Below is an outline of the construction.
            Note that in Case 2
            one needs a little extra argument.

            By Lemma \ref{not-shared},
            we may assume that every JSJ torus of $M$ is adjacent to two distinct
            pieces. We will rewrite $J_0$ as $J_-$,
            and write $J_+\subset M$ be the other JSJ piece adjacent to $T_0$.
            The discussion falls into three cases according to the types
            of the pieces $J_\pm$.

            \medskip\noindent\textbf{Case 1}. If $J_\pm$ are both hyperbolic,
            by \cite[Proposition 3.11]{PW-mixed},
            we may construct two geometrically finite
            subsurfaces $R_\pm$, $\pi_1$-injectively,
            properly immersed in $J_\pm$,
            such that $\partial R_\pm$ are nonempty and
            cover $\zeta_0$. The merging trick allows us
            construct $S$.
            More precisely, we pass to a possibly
            disconnected finite cover $\tilde{R}_\pm$
            of $R_\pm$, so that they have the same number
            of boundary components, and such that
            all components of $\partial \tilde{R}_\pm$
            cover $\zeta_0$ with the same unsigned degree, \cite[Lemma 3.14]{PW-mixed}.
            Then $S$ can be obtained by arbitrarily matching
            up the  boundary components and taking a connected component
            of the result. In this case, we do not need to pass to a further cover of $M$,
            so we take $\tilde{M}$ to be $M$, and $(\tilde{J}_0,\tilde{T}_0,\tilde{\zeta}_0)$
            to be $(J_0,T_0,\zeta_0)$. Set $\tilde{S}=S$, $\tilde{j}:\tilde{S}\looparrowright M$ the immersion
            and $\tilde{R}$  a copy of $\tilde{R}_+$.  The partial PW subsurface in $\tilde{M}$
            can be picked as $(\tilde{S},\tilde{R},\tilde{j})$.

            \medskip\noindent\textbf{Case 2}. If $J_\pm$ are both Seifert fibered,
            we need to recall the antennas property for graph manifolds,
            introduced in the proof of \cite[Proposition 3.1]{PW-graph}.

            For simplicity, let $N$ be a graph manifold
            with nonempty boundary that satisfies Proposition \ref{not-shared}.
            Then in our notations,
            we say that $N$ has the \emph{antennas property}, if
            for any two adjacent JSJ pieces $J_0$, $J_1$,
            there is a chunk $A$ of $N$, called an \emph{antenna},
            which is the union of consecutively adjacent
            distinct pieces $J_0,J_1,\cdots,J_n$ (more precisely, $J_i\cap J_j$ is a JSJ torus if $|j-i|=1$, and is empty otherwise),
            such that $J_n$ contains a boundary component of $N$.

            In our discussion, we consider the maximal graph-manifold chunk $Q\subset M$
            containing $T_0$ as a JSJ torus. It is implied by the proof of
            \cite[Proposition 3.1]{PW-graph} that there is a finite cover $\tilde{M}$
            of $M$, such that any elevation $\tilde{Q}$ has the antennas property.

            Note that $J_\pm\subset Q$. Take elevations $\tilde{J}_\pm\subset \tilde{Q}$
            of $J_\pm$ adjacent along an
            elevation $\tilde{T}_0$ of $T_0$, and
            take an elevation $\tilde{\zeta}_0\subset \tilde{T}_0$
            of $\zeta_0$. For simplicity, still denote $\tilde J_{\pm}$ and so on by  $J_{\pm}$ and so on in this and further coverings.
            We take two antennas $A_\pm$, starting
            with $J_0=J_\mp$ and $J_1=J_\pm$, respectively. Passing to a finite
            cover of $\tilde{M}$ induced by a cover of its dual graph if necessary,
            we may assume $A_+$ and $A_-$ have no common JSJ piece other than $J_\pm$,
            so we call $B=A_+\cup A_-$ a bi-antennas throught $\tilde{T}_0$.
            We may further assume the dual graph of $\tilde{Q}$ has no cycle of at most
            three edges, then there is no JSJ piece of $\tilde{Q}$
            adjacent to two JSJ pieces of $B$.

            Denote $B_\pm$ the two parts of $B$ separated by $\tilde T_0$.
            We proceed to construct a properly embedded, incompressible, and boundary-incompressible
            subsurface $\tilde{E}^*$ such 
            that $\tilde{E}^*$ intersects $\tilde{T}_0$ in slopes parallel to $\tilde\zeta_0$,
            and that $\tilde{T}_0$ cuts $\tilde{E}^*$ into two parts $\tilde{E}^*_\pm\subset \tilde Q$.
            This follows the construction in \cite[Proposition 3.1]{PW-graph},
            as outlined below.
                
             Start with $\tilde\zeta_0$ and try to extend a surface $\tilde E_\pm\subset B_\pm$. Due to the symmetry,
            we just discuss the extension on $B_+$. 
            Suppose first $\tilde\zeta_0$ is not a fiber of $J_1$ up to isotopy.
            One can find a horizontal properly embedded 
            incompressible subsurface $\tilde{E}_+$
            of  $B_+$,
            such that $\tilde{E}_+$ intersects $\tilde{T}_0$ in parallel
            copies of $\tilde{\zeta}_0$.
            Moreover, if a component of $\partial \tilde{E}_+$ does not
            lie on $\partial \tilde{Q}$,
            one can also make sure that it lies on a JSJ torus inside $\tilde{Q}$
            parallel to the adjacent fiber. Take a pair of
            oppositely oriented parallel copies $\tilde{E}_+^\uparrow$ and $\tilde{E}^\downarrow_+$ of $\tilde{E}_+$.
            For each component $c\subset\partial\tilde{E}_+$ that lies in a JSJ torus inside $\tilde{Q}$,
            one can find a properly embedded, boundary-essential vertical
            annulus in the adjacent piece bounding $c^\uparrow\cup c^\downarrow$,
            and glue this vertical annulus to
            $\tilde{E}^\uparrow_+$ and $\tilde{E}^\downarrow_+$ along the boundary,
            correspondingly.
            Since we assumed that no piece of $\tilde{Q}$ is adjacent to
            two JSJ pieces of $B$, the result is a properly embedded, incompressible, and
            boundary-incompressible subsurface $\tilde{E}^*_+$ as desired.
            Suppose then $\tilde\zeta_0$ is a fiber of $J_1$. Then two copies of $\tilde\zeta_0$ with opposite  direction
            bound an essential vertical annulus in $J_1$ as we just discussed, which will be our $\tilde E_+= \tilde E^*_+$.

            In general, $\tilde{E}^*$ is not closed, and $\partial\tilde{E}^*$ intersects
            $\partial\tilde{Q}$ in parallel slopes on each component that it reaches.
            Since $\tilde{Q}$ is a maximal graph-manifold chunk of $\tilde{M}$,
            for each such component as above, there
            is a geometrically finite, $\pi_1$-injectively
            immersed proper subsurface of the adjacent hyperbolic piece whose
             boundary finitely covers the corresponding slope.
            Performing the merging trick again, we obtain a PW subsurface
            $\tilde{j}:\tilde{S}\looparrowright\tilde{M}$,
            so that the part of $\tilde{S}$ inside $\tilde{Q}$ covers $\tilde{E}^*$.
            Note that the virtual embeddedness of $\tilde{S}$ in graph-manifold chunks
            follows from \cite[Lemma 3.6]{PW-mixed}.
            In particular, $\tilde{S}$ intersects $\tilde{T}_0$
            along covers of $\tilde\zeta_0$.
            Moreover, $\tilde{S}$ is cut by $\tilde{T}_0$
            into two parts $\tilde{S}_\pm$, and the part of $\tilde{S}_\pm$ inside $\tilde{Q}$
            covers $\tilde{E}^*_\pm$, respectively. We pick   a connected component
            of $\tilde{S}_+$ for $\tilde{R}$.

            Now the triple $(\tilde{S},\tilde{R},\tilde{j})$ defines a partial PW
            subsurface of $\tilde{M}$ with respect to $(\tilde{J}_0,\tilde{T}_0,\tilde\zeta_0)$.

            \medskip\noindent\textbf{Case 3}. Suppose one of $J_\pm$ is hyperbolic and the other
            is Seifert fibered. This case is a mixture of the previous two cases, and
            the construction is very similar, so we omit the details.
            In fact, this case was also
            covered by the construction of \cite[Section 3]{PW-mixed}, although
            not explicitly stated.
        \end{proof}

    \subsection{Virtual existence of corridor surfaces}
    
    The surfaces provided in the following theorem, later serving as 
    corridor surfaces, will be  crucial in proving Theorem \ref{virt-non-zero}. 

    \begin{theorem}\label{corridorCover}
    	Let $M$ be an orientable closed irreducible mixed $3$-manifold
        and let $\zeta_0$ be a slope on a JSJ torus $T_0$ adjacent to
        a JSJ piece $J_0\subset M$.
       	Suppose $R\looparrowright M$ is a parallel-cutting partial PW subsurface bounded virtually by
        $\zeta_0$ outside $J_0$. Then there exists a regular finite cover $\tilde{X}$
        of the carrier chunk $X(R)$ in which every elevation of $R$ is embedded, intersecting any
        elevation of the carrier boundary $T_0$ in at most one slope.
	\end{theorem}

	The major part of the proof of  Theorem \ref{corridorCover} is the following  weaker version Proposition \ref{YmR}, 
	which is stated in a rather complicated form so that  
	the application of  Przytycki--Wise results, the merging process, and the intersection counting  become more explicit
	in the proof.  To state Proposition \ref{YmR},  we need some terminologies.
  
    For any positive integer $m$, we construct a $2$-complex $Y_m(R)$ immersed in $X(R)$ as follows.
    For each component $c\subset\partial R$, let $\tilde{T}_0^m(c)$ be
    a copy of a cover of $T_0$, in which there are exactly $m$ distinct elevations of $\zeta_0$,
    each of them covering $\zeta_0$ with the same degree as $c$.
    We glue $\tilde{T}_0^m(c)$ with $R$ naturally by identifying an elevation of $\zeta_0$ with
    $c\subset\partial R$. After the gluing for each boundary component of $R$, we get a $2$-complex
    	$$Y_m(R)\,=\,R\,\cup\,\bigcup_{c\subset\partial R}\,\tilde{T}_0^m(c),$$
    and there is a natural immersion
 	   $$Y_m(R)\looparrowright X(R),$$
    which is the immersion of the partial PW subsurface restricted to $R$, and is the covering of the
    boundary torus $T_0$ restricted to the tori $\tilde{T}_0^m(c)$.
    Note that $Y_m(R)$ is naturally defined in the sense that if we chose a different cover
    $\tilde{T}_0^m(c)$ or a different elevation of $\zeta_0$ for the gluing, the resulted immersion
    would differ only by a homeomorphism of $Y_m(R)$. If $R$ is oriented, so is $\partial R$ and its covers.
    
    \begin{proposition} \label{YmR}
    	With  the notations above and the oriented surface $R$,
        there exists an integer $m_0>0$ such that for any positive integral multiple $m$ of $m_0$, there is a regular finite cover $\tilde{X}$
        of $X(R)$ in which every elevation $\tilde Y_m(R)$ of $Y_m(R)$ is embedded, 
        and moreover $\tilde {\mathcal R}\cap\tilde T_0$ consists of  $r$  directed parallel circles 
        in $\tilde T_0$ induced 
        from  $\partial R$,  where $\tilde {\mathcal R}$ is the union of elevations of $R$ contained in $\tilde Y_m(R)$, $\tilde T_0$ 
        is an elevation of $T_0$ contained in $\tilde Y_m(R)$, and  $r$ is a positive integer.
    \end{proposition}

	To prove Proposition \ref{YmR}, we need the following  Lemma \ref{horizontalParallelCutting}.

    There is a canonical
    (possibly disconnected) compact essential subsurface of $R$,
    called the \emph{horizontal part}.
    It is the union of all subsurfaces of $R$ that are properly horizontally 
    immersed  in Seifert fibered pieces, glued up along the cut curves
    where any two are adjacent.
    Note that every complementary component of the union of the horizontal
    part and the cut curves
    is either a vertical cylinder immersed in a Seifert fibered piece,
    or a geometrically finite cusped subsurface immersed in a hyperbolic piece.

    \begin{lemma}\label{horizontalParallelCutting}
    	For every component $F$ of the horizontal part of $R$, there is a unique
    	graph-manifold chunk $Q_F\subset X(R)$ in which $F$ is properly horizontally
    	immersed. Moreover, suppose $\tilde{Q}_F$ is a regular finite cover of $Q_F$
    	in which any elevation of $F$ is embedded,
    	then for any elevation $\tilde{T}$ of
    	a component $T\subset\partial Q_F$ and any elevation $\tilde{F}$ of
    	$F$, the number of components of $\tilde{T}\cap \tilde{F}$ depends only on $T$ and $F$.
	\end{lemma}

    \begin{proof} 
    	The first claim is direct:
        let $Q_F$ be the minimal subchunk of $X(R)$  containing the image of $F$, which is unique.
        Since $R$ is parallel cutting, it is clear that $\partial F$ cannot sit in any JSJ
        torus in the interior of $Q_F$, so $F$ is properly horizontally immersed in $Q_F$.
        Below we devote  to the proof the "moreover" part.

            Since $\tilde{Q}_F$ is a regular cover of $Q_F$,
            it suffices to fix an elevation $\tilde{F}$ and show
            that the number of components of $\tilde{T}\cap\tilde{F}$ is constant
            for all elevations $\tilde{T}$ of  $T$.
            Let $f\subset T$ be a Seifert fiber of the adjacent JSJ piece of $Q_F$.
            As $R$ is parallel cutting, there is also a slope $s\subset T$
            covered by all the components of $\partial F$ that are immersed in $T$.
            For any elevation $\tilde{T}$ of $T$,
            we may pick elevations $\tilde{f},\tilde{s}\subset\tilde{T}$ of $f$ and $s$, respectively.

            The (geometric) intersection numbers $i(\tilde{f},\tilde{s})$ and $i(f,s)$ are related
            by the formula
                $$i(\tilde{f},\tilde{s})\,=\,i(f,s)\cdot\frac{[\tilde{f}:f]\,[\tilde{s}:s]}{[\tilde{T}:T]},\qquad (4.1)$$
            where $[-:-]$ denotes the covering degree. 
            This follows from 
            $p^{-1}(f)=\frac{[\tilde T:T]}{[\tilde f:f]}\tilde f$, $p^{-1}(s)=\frac{[\tilde T:T]}{[\tilde s:s]}\tilde s$ and
            $\,i(p^{-1}(f),p^{-1}(s))=[\tilde T:T]\,i(f,s)$,  where $p:\tilde T\to T$ is the discussed covering. 
            
            Since $\tilde{F}$ is horizontally embedded,
            $\tilde{Q}_F$ fibers over the circle with fiber $\tilde{F}$, each components of $\tilde{T}\cap\tilde{F}$ must be a copy of $\tilde s$.
            Hence the number of components of $\tilde{T}\cap\tilde{F}$ satisfies
                $$|\pi_0(\tilde{T}\cap\tilde{F})|\,=\,\frac{i(\tilde{f},\tilde{F})}{i(\tilde{f},\tilde{s})}.$$
            Thus it suffices
            to show $i(\tilde{f},\tilde{F})$ is constant for all elevations $\tilde{f}$ of $f$.
            
            By a calculation similar to formular (4.1) we have 
            $$i(\tilde{f},\tilde{F})=i(\tilde{f},\tilde{\partial F})=i(\tilde{f}, p^{-1}(s))=i({f},s)[\tilde f:f] \qquad (4.2).$$

            Let $\Lambda$ and $\tilde\Lambda$ be the dual graph associated
            to the JSJ decompositions of $Q_F$ and $\tilde{Q}_F$. Note that there is a natural
            combinatorial map $\tilde\Lambda\to\Lambda$ induced by the covering.
            For any vertex $\tilde{v}$ of $\tilde\Lambda$,
            we write the corresponding JSJ piece of $\tilde{Q}_F$ as $\tilde{J}_{\tilde{v}}$, and the ordinary Seifert fiber
            of $\tilde{J}_{\tilde{v}}$ as $\tilde{f}_{\tilde{v}}$; for any edge $\tilde{e}$ of $\tilde\Lambda$,
            we write the corresponding JSJ torus as $\tilde{T}_{\tilde{e}}$, and the slope of $\tilde{T}_{\tilde{e}}$
            parallel to the components of $\tilde{T}_{\tilde{e}}\cap\tilde{F}$ as $\tilde{s}_{\tilde{e}}$.
            The notations for $\Lambda$ are similar. As $R$ is parallel cutting,
            for any directed edge $e$ of $\Lambda$, the ratio
                $$\lambda_{e}=\frac{i(f_{\mathrm{ter}(e)},s_e)}{i(f_{\mathrm{ini}(e)},s_e)}$$
            is a positive rational number
            depending only on $F$ and $T_e$. Here $\mathrm{ini}(e)$, $\mathrm{ter}(e)$
            denotes the initial vertex and the terminal vertex of $e$, respectively.
            Suppose $\tilde{e}_1,\cdots,\tilde{e}_n$
            is a sequence of edges of $\tilde\Lambda$, consecutively joining the sequence of vertices
            $\tilde{v}_0,\cdots,\tilde{v}_n$ of $\tilde\Lambda$. We write $v_k$, $e_k$ for the
            image of $\tilde{v}_k$, $\tilde{e}_k$ under $\tilde\Lambda\to\Lambda$, respectively.
            From the formula (4.2), we have:
            \begin{eqnarray*}
                \frac{i(\tilde{f}_{\tilde{v}_n},\tilde{F})}{i(\tilde{f}_{\tilde{v}_0},\tilde{F})}&=&
                \prod_{k=1}^{n}\frac{i(\tilde{f}_{\tilde{v}_k},\tilde{F})}{i(\tilde{f}_{\tilde{v}_{k-1}},\tilde{F})}\\
                &=&
                \prod_{k=1}^{n}\,
                \frac{i(f_{v_k},s_{e_k})}{i(f_{v_{k-1}},s_{e_k})}\cdot
                \frac{[\tilde{f}_{\tilde{v}_k}:f_{v_k}]}{[\tilde{f}_{\tilde{v}_{k-1}}:f_{v_{k-1}}]}\\
                &=&
                \frac{[\tilde{f}_{\tilde{v}_n}:f_{v_n}]}{[\tilde{f}_{\tilde{v}_{0}}:f_{v_{0}}]}\cdot
                \prod_{k=1}^{n}\lambda_{e_k}.
            \end{eqnarray*}
            In particular, if $\tilde{f}$ and $\tilde{f}'$ are two elevations of the Seifert fiber $f$ on
            a given component $T\subset\partial Q_F$, we may pick a path as above so that $\tilde{f}_{\tilde{v}_0}=
            \tilde{f}$ and $\tilde{f}_{\tilde{v}_n}=\tilde{f}'$. Since $v_0=v_n$ and $\tilde{Q}_F$ is a regular cover,
            $[\tilde{f}_{\tilde{v}_n}:f_{v_n}]=[\tilde{f}_{\tilde{v}_{0}}:f_{v_{0}}]$.
            Thus it suffices to show $\lambda_{e_1}\cdots
            \lambda_{e_n}=1$ for any cycle $e_1,\cdots,e_n$ of $\Lambda$.

            To see this, note that $F$ is properly horizontally immersed in $Q_F$. If $e$ is a directed edge of $\Lambda$,
            then any component of $j^{-1}(J_{\mathrm{ini}(e)})\subset F$ is adjacent to a component of
            $j^{-1}(J_{\mathrm{ter}(e)})\subset F$. Therefore, if $e_1,\cdots,e_n$ is a cycle of $\Lambda$,
            then starting with any component $C\subset j^{-1}(J_{v_0})$, we may find a path $\gamma:[0,1]\looparrowright F$
            so that $\gamma(0)$ lies in $C$, consecutively intersects $T_{e_1},\cdots,T_{e_n}$, and $\gamma(1)$ lies
            in a component $C'\subset j^{-1}(J_{v_0})$. Since there are only finitely many components of $j^{-1}(J_{v_0})$,
            we may join a number of such $\gamma$'s as above to obtain
            a loop $S^1\looparrowright F$ which goes around the cycle for a positive number of times, say $r$ times.
            Because $F\looparrowright Q_F$ is a proper horizontal immersion which is a virtual embedding,
            it follows from the criterion
            of Rubinstein--Wang \cite[Theorem 2.3]{RW} that
                $$\left(\lambda_{e_1}\cdots\lambda_{e_n}\right)^r\,=\,1.$$
            Therefore, $\lambda_{e_1}\cdots\lambda_{e_n}=1$ and this completes the proof.
        \end{proof}
        
        \begin{proof}[{Proof of Proposition \ref{YmR}}] 
        
        Since one cannot directly claim that $Y_m(R)$ is $\pi_1$-injectively immersed in $X(R)$ 
        (indeed this is not necessarily true in general), 
        we will  prove the proposition by the following strategy:
            First  for virtually all positive integer $m$, by using the results of Przytycki and Wise, as well as  
            Lemma \ref{horizontalParallelCutting},
            we can virtually embed $Y_m(R)$ into a compact $3$-manifold
            $\mathcal{Y}^*_m$ and  make sure that $\mathcal{Y}^*_m\looparrowright X(R)$ is virtually embedded when restricted on
              each JSJ piece.  
             And finally we will apply the merging trick (Proposition \ref{mergeFiniteCovers}) to get the global embedding in the Proposition \ref{YmR}.
            For simplicity, we write $X$ for $X(R)$ and $Y_m$ for $Y_m(R)$.

            For the triple $(S,R,j)$
            associated to the partial PW subsurface $R\looparrowright M$, there is a regular finite
            cover of $M$ in which any elevation of $S$ is embedded, by the separability
            of PW subsurfaces (Lemma \ref{PWseparability}). Hence we may assume $X'$ is a regular
            finite cover of $X$ in which any elevation of $R$ is embedded. Let
                $$R'\subset X'$$
            be an elevation of $R$.
            Take a compact regular neighborhood of the horizontal part of $R'$,
            and for each cut curve or boundary curve not adjacent to the horizontal
            part, take a compact regular neighborhood of it,
            and make sure these regular neighborhoods are mutually disjoint.
            Let $\mathcal{F}'\subset R'$ be the union of these regular neighborhoods.
            For each component $F'\subset\mathcal{F}'$, let $\mathcal{Q}'_{F'}\subset X'$
            be a compact regular neighborhood of a chunk, or of a JSJ or boundary torus,
            so that $F'$ is properly embedded in $\mathcal{Q}'_{F'}$.
            Hence $\mathcal{Q}'_{F'}$ is a bundle over the circle with  fibre $F'$, denoted as $(F', \phi_{F'})$, where $\phi_{F'} : F'\to F'$ 
            is the gluing map, which is a periodic on each reducible piece in the sense of Nielsen-Thurston.

            By Lemma \ref{horizontalParallelCutting} and some straighforward verification,
            for a JSJ or boundary torus $T\subset X$ and
            a component $F'\subset\mathcal{F}'$, any elevation $T'$ of $T$ intersects $F'$
            if and only if $T'\subset\mathcal{Q}'_{F'}$; and furthermore,
            such $T'$ intersects $F'$ in a number $\mu_T(F')>0$ of components, depending
            only of $T$ and $F'$ (indeed, only on the subsurface of $R$ that $F'$ covers).
            As $R$ is parallel cutting,
            let $s\subset T$ be the slope covered by the components of $F'$.
            Note that $s$ is $\zeta_0$ if $T$ is $T_0$.
            Let
                $$k'_T=\frac{[T':T]}{[s':s]}$$
            be the number of elevations of $s$ in any elevation of $T$,
            which is well defined since $X'$ is a regular
            finite cover of $X$.
            Let
                $$m'_0>0$$
            be the product of the least common multiple of all $k'_T$ and the least common multiple
            of all $\mu_T(F')$.
            For any positive multiple $m$ of $m'_0$,
            let $p_*: \mathcal{Q}^*_{F'}\to \mathcal{Q}'_{F'}$ be the cyclic covering dual to $F'$, of degree
                $$d=[\mathcal{Q}^*_{F'}:\mathcal{Q}'_{F'}]\,=\,m\cdot\frac{\mu_{T_0}(F')}{k'_{T_0}},$$
            or just $m$ if $\mathcal{Q}^*_{F'}$ contains no elevations of $T_0$.
           
            Clearly $\mathcal{Q}^*_{F'}=(F', \phi_{F'}^d)$.
           Since $d$ is a multiple of each $\mu_{T_0}(F')$, 
           each component of $\partial F'$ (also cutting curves) is invariant under $\phi_{F'}^d$.
           Therefore each component of $\partial\mathcal{Q}^*_{F'}$ contains one and only one   component of $\partial F'$.
           Note that $p^{-1}_*(F')$ has $d$ components in $\mathcal{Q}^*_{F'}$ and each one is a lift of $F'$;
           moreover  for each $T_0'$ covering $T_0$, $p^{-1}_{*}(T'_0)$ has $\mu_{T_0}(F')$ components.
           Let $ {p_0} : T_0' \to T_0$ be the discussed covering, then ${p_0}^{-1}(\zeta_0)$ has $k'_{T_0}$ components  in $T_0'$, therefore 
           ${(p_*\circ p_0)}^{-1}(\zeta_0)$ has $dK'_{T_0}$ components on $\partial \mathcal{Q}^*_{F'}$. 
            It follows that in every component
            of $\partial\mathcal{Q}^*_{F'}$ that covers $T_0$, there are exactly
            $\frac{dK'_{T_0}} {\mu_{T_0}(F')}=m$ elevations of $\zeta_0$,

            Let $\mathcal{W}'_{R'}\subset X'$ be a compact regular neighborhood of $R'$,
            and let $\mathcal{W}^*_{F'}\subset\mathcal{Q}^*_{F'}$ be a compact regular neighbood
            of a fiber $F'$. Since $\mathcal{W}'_{R'}\cap\mathcal{Q}'_{F'}$ is naturally homeomorphic
            to $\mathcal{W}^*_{F'}$, we may take a copy of $\mathcal{W}'_{R'}$ and glue it with
            a copy of $\mathcal{Q}^*_{F'}$ by identifying $\mathcal{W}'_{R'}\cap\mathcal{Q}'_{F'}$
            and $\mathcal{W}^*_{F'}$, for all components $F'\subset\mathcal{F}'$.
            The result is a compact $3$-manifold 
            	$$\mathcal{Y}^*_m\,=\,\mathcal{W}_{R'}'\,\cup\,\bigcup_{F'\subset\mathcal{F}'}\,\mathcal{Q}^*_{F'}$$
            with boundary, and there is a natural immersion
                $$\varphi:\,\mathcal{Y}^*_m\looparrowright X',$$
                which sends  $W_{R'}'$ to $W_{R'}'\subset X'$ by the identity and each  $\mathcal{Q}^*_{F'}$ to $\mathcal{Q}'_{F'}\subset X'$
                via the given covering.
            Moreover, the union of the copy $R'\subset\mathcal{Y}^*_m$ and
            all components of $\partial\mathcal{Q}'_{F'}$
            that cover $T_0$ is an embedded $2$-complex
                 $$Y^*_m\subset\mathcal{Y}^*_m.$$
            Moreover, $Y^*_m$ naturally covers $Y_m$ with degree $[R':R]$, and the hanging tori $\tilde{T}_0^m(c)$
            all lift. In other words,
            $Y_m$ is virtually embedded in $\mathcal{Y}^*_m$, naturally in the sense that
            the map  $Y^*_m\looparrowright X'$ induced from the immersions of $Y_m$ and of
            $\mathcal{Y}^*_m$ are the same up to homeomorphism of $Y^*_m$.
            
            We are going to show that $\varphi:\,\mathcal{Y}^*_m\looparrowright X'$ is a virtual embedding. To do this 
            we first show that the restriction on each JSJ piece of $\mathcal{Y}^*_m$ is a virtual embedding.
            
            The JSJ tori of $\mathcal{Y}^*_m$ are exactly the JSJ or boundary tori of
            all $\mathcal{Q}^*_{F'}$. The JSJ pieces of $\mathcal{Y}^*_m$ are the JSJ pieces
            of $\mathcal{Q}^*_{F'}$, and the pieces containing the components of
            $R'\setminus\mathcal{F'}$. 
            To describe the latter type
            more precisely, consider any connected compact subsurface $V'\subset R'$,
            which is properly immersed in a JSJ piece of $X'$ and vertically or geometrically finite.
            For each such $V'$, there is a unique component of
            $R'\setminus\mathcal{F}'$ contained in $V'$, which is
            isotopic to the interior of $V'$. Then the unique JSJ piece
            of $\mathcal{Y}^*_m$ containing $V'$
            is the piece bounded by all the JSJ tori
            adjacent to $\partial V'$, and this piece deformation retracts to
            the union of $V'$ and all its adjacent JSJ tori, denoted by $Y^*(V')$.
            In fact, $Y^*(V')$ can be described in a similar fashion as that of $Y_m(R)$.
            For each component $c'\subset\partial V'$ that covers a slope in a torus $T\subset X$,
            we glue a copy of a cover $\tilde{T}^{m(c')}(c')$ of $T$ to $V$ along $c$ in a similar way
            as  $\tilde{T}^m_0(c)$. Explicitly, since there
            is a unique $F'\subset\mathcal{F}'$ containing $c'$,
                $$m(c')\,=\,[\mathcal{Q}^*_{F'}:\mathcal{Q}'_{F'}]\cdot\frac{k'_T}{\mu_T(F')},$$
            which is a positive integer by our choice of $m'_0$.

            Suppose $V'\subset R'$ is a vertical or geometrically finite subsurface as above,
            and $K^*\subset\mathcal{Y}^*_m$ is the unique JSJ piece containing $V'$. As we have seen,
            the inclusion of the $2$-complex $Y^*(V')\hookrightarrow K^*$ is a homotopy equivalence.
            Let $J'\subset X'$ be the JSJ piece in which $K^*$ is immersed into, so that there is an induced
            immersion $\varphi|:Y^*(V')\looparrowright J'$. If $J'$ is hyperbolic, then
            by \cite[Theorem 4.1]{PW3}, $\varphi$ restricted to $Y^*(V')$ is $\pi_1$-injective
            and relatively quasiconvex
            if all $m(c')$ above are sufficiently large; and in this case,
            it is a consequence of the relative quasiconvex separability due to
            Wise \cite[Theorem 16.23]{Wise-long}
            (cf.~\cite[Corollary 4.2]{PW3}) that $\pi_1(Y^*(V'))$ is indeed separable.
            Then we may find a finite cover of $J'$ in which the elevations of $Y^*(V')$,
            and hence elevations of $K^*$, are embedded. 
            
            If $J'$ is Seifert fibered, then $J'$ is a product and $Y^*(V')$ is
            just the union of a properly immersed (boundary-essential) vertical annulus
            together with covers of the tori
            that are adjacent to.
            If all $m(c')$ are sufficiently large,
            one can easily see that $\pi_1(Y^*(V'))$ is embedded in $\pi_1(J')$ and
            is separable. Then we may again find a finite cover of $J'$ in which elevations of $Y^*(V')$,
            and hence elevations of $K^*$, are embedded. From the formula of $m(c')$ above,
            it is clear that $m(c')$ can be arbitrarily as large as desired
            if $m$ is sufficiently large.
            Therefore, we may pick
                $$m_0>0$$
            to be a sufficiently large multiple of $m'_0$, so that any multiple of $m$
            is sufficiently large to ensure that the $\pi_1$-injectivity and separability
            of $Y^*(V')$ work.

            Note that if $K^*$ is a JSJ piece of $\mathcal{Y}^*_m$ that is contained in some $\mathcal{Q}^*(F')$,
            it covers a JSJ piece $J'$ of $X'$.
            
            Therefore, we have shown that for every JSJ piece $K^*\subset\mathcal{Y}^*_m$,
            there is a  JSJ piece $J'$ of $X'$ that contains the immersed image of $K^*$, and moreover 
            we have an embedding $\phi'': K''\to J_{K^*}$ which covers $\varphi |: K^*\to J'$.
            Now for each $J_i'\subset X'$, if $J'_i$ contains the image of  a JSJ piece $K^*\subset\mathcal{Y}^*_m$,
            let $J^*_i$ be the common finite cover of all those $J_{K^*}$, 
            and otherwise set $J^*_i=J'_i$.
             By Proposition \ref{mergeFiniteCovers}, we have a regular finite JSJ $l$-characteristic covering $p'': X''\to X'$
             such that each covering $J''\to J'_i$ factor through $J^*_i$.
             
             Let $\varphi'': \mathcal{Y}''_m\looparrowright X''$ be any elevation of  $\varphi:\,\mathcal{Y}^*_m\looparrowright X'$.
            Since the virtual embeddedness is preserved under passage
            to further covers, $\varphi''|$ on each JSJ piece is
            an  embedding.  It follows that 
            the induced map on the dual graph $\Lambda(\mathcal{Y}''_m)\to\Lambda(X'')$
            is a combinatorial local embedding, which is $\pi_1$-injective. Because
            $\pi_1(\Lambda(X''))$ is a free group, and hence is LERF, 
            it has a regular
            finite cover in which any elevation of $\Lambda(\mathcal{Y}''_m)$ is an embedded subgraph.
            Therefore, we have  a regular finite JSJ 1-characteristic covering $p''': X'''\to X''$ so that 
            any elevation $\varphi''': \mathcal{Y}'''_m\looparrowright X'''$  of  $\varphi:\,\mathcal{Y}^*_m\looparrowright X'$
            is an embedding. 
            As we discussed in Proposition \ref{mergeFiniteCovers} and its remark, 
            by passing to a further finite cover $\tilde X$, we can assume that 
            the JSJ $l$-characteristic covering  $\tilde X\to X'$
            is characteristic in the usual sense, 
            which implies that the covering $\tilde X\to X'\to X$ is a finite regular covering. 

            In conclusion, for any positive multiple $m$ of the $m_0$ we have chosen,
            there is a regular finite cover $\tilde{X}$ of $X$, in which any elevation
            of $\mathcal{Y}^*_m$, and hence any elevation $\tilde Y$ of $Y_m$, is embedded.

            Now we are going to prove the ``moreover part'':        We fix an orientation of $R$. Let $m$ be a positive integer ensured by
            the first half of Proposition \ref{YmR}, so that $Y_m(R)\looparrowright X(R)$ is
            a virtual embedding. We assume $\tilde X$ is a regular finite cover of $X(R)$ constructed 
            in Proposition \ref{YmR} in which any elevation $\tilde Y$ of $Y_m(R)$ is embedded. 
            
            Note that there is a copy
            of $R$ contained in $Y_m(R)$, so for any elevation $\tilde Y\subset \tilde X$,
            there is a collection of mutually disjoint, embedded elevations of $R$, with naturally induced
            orientations.
            We fix an elevation $\tilde Y$ of $Y$, and
            let $\tilde{ \mathcal{R}}\subset \tilde Y$ be the union of elevations of $R$ contained in $\tilde Y$.
            
            Since $\zeta_0$ has  $ \frac{[\tilde T_0:T_0]}{[\tilde \zeta_0:\zeta_0]}$ elevations
             in any elevation $\tilde T_0$ of $T_0$, and  $\zeta_0$ has $m$ elevations in $T_0^m(c)\subset Y_m(R)$ for each component $c$
            of $\partial R$, it follows that $c\subset T_0^m(c)\subset Y_m(R)$ has $r= \frac{[\tilde T_0:T_0]}{m[\tilde \zeta_0:\zeta_0]}$ 
            elevations
            in $\tilde T_0$. Since $R\subset Y_m(R)$ meets $T_0^m(c)$ exactly on $c$, it follows that 
             for any elevation $\tilde T_0$ of $T_0$
            contained in $\tilde Y$,
            there are exactly $r$ components
            of $\partial\tilde{\mathcal{R}}\cap \tilde T_0$.
            Furthermore, it is clear
            from the construction of $Y_m(R)$
            that for any $\tilde T_0\subset \tilde Y$, all components
            of $\partial\tilde {\mathcal{R}}\cap \tilde T_0$ cover the same
            component of $\partial R$, and in particular, they are
            directly parallel on $\tilde T_0$
            with the direction induced from $\partial\mathcal{R}$.
            Note also that for any $\tilde T_0$ not contained in $\tilde Y$,
            $\tilde T_0\cap\tilde {\mathcal{R}}$ is the empty set.
           
            This completes  the proof of Proposition \ref{YmR}.
		\end{proof}
            
        \begin{proof}[{Proof of Theorem \ref{corridorCover}}]
        	We start from the conclusion of Proposition \ref{YmR}.
            To match the notations, still denote by $X'$, $\mathcal R'$, $Y'$ and $T'_0$ the spaces  $\tilde X$, $\tilde {\mathcal{R}}$, $\tilde Y$ and $\tilde T_0$ obtained in Proposition \ref{YmR}, where everything is oriented.

            The oriented properly embedded subsurface ${\mathcal R}'$
            represents a class $[{\mathcal R}']\in H_2(X',\partial X';\ZZ)$.
            Then the homological pairing with $[\mathcal R']$
            induces a quotient homomorphism
                $$l_{\mathcal{R}'}:\,
                 H_1(X';\,\ZZ)  \stackrel{[\mathcal R']}{\longrightarrow}  \ZZ \to \ZZ_r.$$

           Denote the group of deck transformations of the covering
            of $X'\to X$ as $\Gamma_{X'}$ and 
            by taking the direct sum of all $\tau^*(l_{\mathcal{R}'})$, where $\tau$ runs
            over $\Gamma_{X'}$ we define a homomorphism of integral modules:
                $$L=\oplus_{\tau\in \Gamma_{X'}} {\tau^*}(l_{\mathcal R'}) 
                =\oplus_{\tau\in \Gamma_{X'}}l_{\tau(\mathcal R')} :\,H_1(X';\ZZ)\,\to\,\ZZ_r^{\oplus \mathrm{Gal}(X')}.$$
            The kernel of the homomorphism of groups
                $$\kappa': \pi_1(X')\longrightarrow H_1(X';\ZZ)\stackrel{L}{\longrightarrow} \ZZ_r^{\oplus\mathrm{Gal}(X')}$$
            is invariant under the deck transformation group $\Gamma_{X'}$, thus
            it follows that $\mathrm{Ker}(\kappa')\subset \pi_1(X')\subset \pi_1(X)$ is a normal subgroup,
            so it defines a regular finite cover
                $$\kappa:\tilde{X}\to X,$$
            which factors through $X'$.

            It remains to verify that every elevation $\tilde{R}\subset\tilde{X}$ of $R$ intersects
            any elevation $\tilde{T}_0\subset\tilde{X}$ of
            $T_0$ in at most one components. Since $\tilde{X}$ is a regular cover, we may assume $\tilde{R}$
            is an elevation of a component $R'\subset\mathcal{R}'$. Thus, an elevation $\tilde{T}_0$
            of $T_0$ intersects $\tilde{R}$ if and only if it covers an elevation ${T}'_0\subset X'$
            of $T_0$ contained in $Y'$. 
               If there were at least two components of $\tilde{T}_0\cap\tilde{R}$ then
            we could pick two points $\tilde{x},\tilde{y}$ on two distinct components,
            and there would be a directed loop $\tilde\alpha$
            formed by two consecutive directed paths
            $\tilde\alpha_{\tilde{T}_0}\subset \tilde{T}_0$ and
            $\tilde\alpha_{\tilde{R}}\subset \tilde{R}$, both joining $\tilde{x}$ and $\tilde{y}$. 
            Because $\tilde{R}$ and $\tilde{T}_0$ cover $R'$ and $T'_0$, respectively,
            and $\tilde R\subset \tilde X$ is a two-sided proper embedded surface, 
            we may perturb $\tilde\alpha$ a bit so that $\tilde \alpha_{\tilde{R}}$ is projected into $X'$ missing the interior of $\mathcal{R}'$.
            Because the algebraic intersection number of  $\tilde R$ and $\tilde \alpha$  
            is always an integer    multiple of $r$,       
                        it follows that the path $\tilde\alpha_{\tilde{T}_0}$
            is immersed under the
            covering into $T'_0$,  and has the algebraic intersection number with  the components of $R'\cap T'_0$ an integral multiple of $r$.
             Because there are exactly $r$
            components of $R'\cap T'_0$, directly parallel on $T'_0$,
           this means that $\tilde{x}$ and $\tilde{y}$ are projected to the same component
            of $T'_0\cap R'$. Up to homotopy, we may assume they are the same, so
            $\tilde{\alpha}_{\tilde{T}_0}$ is the lift of a closed path $\alpha'_{T'_0}\looparrowright T'_0$.
            However, since $\mathcal{R'}$ intersects any elevation $T'_0$ in either the empty set or
            exactly $r$ directly parallel components, $L$ vanishes on $H_1(T'_0;\ZZ)$ for any $T'_0$.
            In other words, every $T'_0$ lifts into $\tilde{X}$.
            This means that the closed path $\alpha'_{T'_0}$ lifts
            into $\tilde{T}_0$ as well, so $\tilde{x}$ and $\tilde{y}$ are the same. This contradicts
            the assumption that they lie on distinct components of $\tilde{R}\cap\tilde{T}_0$.
            We conclude that $\tilde{X}$ is the regular finite cover as desired.
            \end{proof}

\section{Virtual extension of representations}\label{Sec-virtualExtensionOfRepresentations}
	In this section, we construct virtual extension 
	of a representation $\rho_0:\pi_1(J_0)\to \mathscr{G}$ 
	of a JSJ piece $J$ of a mixed $3$-manifold $M$
	assuming that the representation $\rho_0$ has nontrivial kernel 
	on $\pi_1(T)$ for each torus $T\subset \partial J_0$ (Theorem \ref{virtualExtensionRep}).
	For the sake of generality, we abstract a property of the target group $\mathscr{G}$ 
	called class invertibility (Definition \ref{classInvertible}),
	with which we can ``flip'' $\rho_0$ up to conjugation.
	In particular, ${\rm PSL}(2;\C)$ and ${\rm Iso}_e\t{{\rm SL}_2(\R)}$
	are both class invertible (Lemma \ref{geometricClassInv}).

	\begin{definition}\label{classInvertible}
        Let $\mathscr{G}$ be a 
        group,
        and  $\{\,[A_i]\,\}_{i\in I}$ be a collection
        of conjugacy classes of abelian subgroups.
        By a \emph{class inversion} with respect to $\{\,[A_i]\,\}_{i\in I}$,
        we mean an outer automorphism
        $[\nu]\in\mathrm{Out}(\mathscr{G})$,
        such that for any representative abelian subgroup $A_i$
        of each $[A_i]$, there is a representative automorphism
        $\nu_{A_i}:\mathscr{G}\to\mathscr{G}$ of $[\nu]$
        that preserves $A_i$, taking every
        $a\in A_i$ to its inverse.
        We say $\mathscr{G}$ is \emph{class invertible}
        with respect to $\{\,[A_i]\}_{i\in I}$,
        if there exists class inversion.
        We often ambiguously call any collection
        of representative abelian subgroups $\{\,A_i\,\}_{i\in I}$
        a class invertible collection, and call any representative automorphism
        $\nu$ a class inversion.
    \end{definition}

    \begin{theorem}\label{virtualExtensionRep}
        Let $\mathscr{G}$ be a
        group, and
        $M$ be an irreducible orientable closed mixed $3$-manifold.
        For a geometric piece $J_0\subset M$,
        suppose a representation
            $$\rho_0:\pi_1(J_0)\to\mathscr{G}$$
        satisfies the following:
        \begin{itemize}
        \item for every boundary torus $T\subset \partial J_0$, $\rho_0$ has nontrivial kernel
        restricted to $\pi_1(T)$; and
        \item for all boundary tori $T\subset\partial J_0$, $\rho_0(\pi_1(T))$
        form a class invertible collection of abelian subgroups of $\mathscr{G}$.
        \end{itemize}
        Then there exist a finite
        regular cover
            $$\kappa:\tilde{M}\to M,$$
        and a representation
            $$\tilde{\rho}:\pi_1(\tilde{M})\to\mathscr{G},$$
        satisfying the following:
        \begin{itemize}
        \item for one or more elevations $\tilde{J}_0$ of $J_0$,
        the restriction of $\tilde{\rho}$ to $\pi_1(\tilde{J}_0)$ is,
        up to a class inversion, conjugate
        to the pull-back $\kappa^*(\rho_0)$; and
        \item for any elevation $\tilde{J}$ other than the above, of any geometric piece
        $J$, the restriction of $\tilde{\rho}$ to $\pi_1(\tilde{J})$ is cyclic,
        possibly trivial.
        \end{itemize}
    \end{theorem}

    The rest of this section is devoted to the proof of Theorem \ref{virtualExtensionRep}.
    In Subsection \ref{Subsec-coloredMerging}, we construct a cover of $M$
    by merging colored chunks. In Subsection \ref{Subsec-virtualRepresentation},
    each colored chunk will be endowed naturally with a representation,
    up to conjugation. Then the colored merging gives rise
    to a virtual representation as desired.
    There will be three colors $0$ (null), $+1$ (positive), and $-1$ (negative),
    to be assigned to the boundary components of the chunks in our construction.
    To keep in mind, the null color
    will mean that the restricted representation to the boundary is trivial,
    and the signed colors will mean that
    the restricted representation to the boundary is nontrivial
    and the sign indicates whether a class inversion will be applied.

    \subsection{Colored chunks and colored merging}\label{Subsec-coloredMerging}
        Let $M$ be an orientable closed irreducible mixed $3$-manifold containing no essential
        Klein bottles. Let
            $$J_0\subset M$$
        be a selected JSJ piece of $M$. The boundary
        of $J_0$ is a disjoint union of tori:
            $$\partial J_0\,=\,\bigsqcup_{i=1}^{s}\,T_i,$$
        and for each $T_i$, let
            $$\zeta_i\subset T_i$$
        be a selected slope. Fix a direction for each $\zeta_i$.
        With this data, we construct a regular finite
        cover $\tilde{M}$ of $M$ by merging colored
        chunks as follows.

        By Theorem \ref{virtualPartialPW}, for each $T_i\subset M$, there is
        a finite cover $M'_i$ of $M$, and an elevation $(J_{0,i}',T_i',\zeta_i')$
        of the triple $(J_0,T_i,\zeta_i)$, and there is a parallel-cutting partial
        PW subsurface $R'_i\looparrowright M'_i$ virtually bounding
        $\zeta'_i$ outside $J_{0,i}'$. By Theorem \ref{corridorCover}, there is
        a regular finite cover $X''_i$ of the carrier chunk $X(R'_i)\subset M'_i$,
        in which any elevation of $R'_i$ is properly embedded, intersecting each elevation
        of the carrier boundary $T'_i$ in at most one slope.
        Pick an elevation $R^*_i\subset X''_i$ of $R'_i$, and let $X^*_i\subset X''_i$ be the
        carrier chunk of $R^*_i$, namely, the minimal chunk containing $R^*_i$.
        Taking a copy of $X^*_i$ together with $R^*_i$, we call the (abstract) duple
        $(X^*_i,R^*_i)$
        a \emph{corridor chunk} associated to the sloped boundary $(T_i,\zeta_i)$.
        The properly embedded subsurface $R^*_i$ is called the \emph{corridor surface}
        of $X^*_i$, and the \emph{corridor boundary} $\partial^*X^*_i$ of $X^*_i$ is the
        union of boundary components that intersect $R^*_i$. Note that every corridor boundary
        component is an elevation of $T_i$.
        To color the boundary of $X^*_i$, we pick an orientation of $R'_i$.
        For any corridor boundary component $T^*\subset\partial^*X^*_i$,
        $R^*_i$ intersects $T^*$ in exactly one slope $c^*$,
        and $c^*$ has one direction induced from the direction
        of $\zeta_i$, and another direction induced from the orientation of $R^*_i$.
        We color any corridor boundary component
        $T^*$ by $+1$ if these two induced directions of $c^*$ agree,
        or by $-1$ otherwise. We color any non-corridor boundary
        component of $X^*_i$ by $0$. The result is a
        \emph{colored corridor chunk}
            $$(X^*_i,R^*_i).$$
        Note that $X^*_i$ has the same number
        of positively colored boundary components and negatively colored components.

        A copy of $J_0$ with boundary components colored all by $+1$, or all by $-1$, is called
        a \emph{positively colored $J_0$ piece}, or a \emph{negatively colored $J_0$ piece}, respectively.

        A copy of any JSJ piece $J\subset M$ (possibly $J_0$)
        with all boundary components colored by $0$ is called a \emph{null colored JSJ piece}.

        By an \emph{elevated colored chunk}, we mean a finite cover
        of any of the following:
        \begin{itemize}
            \item a positively or negatively colored $J_0$ piece
            \item a colored corridor chunk associated to a $(T_i,\zeta_i)$
            \item a null colored JSJ piece from $M$
        \end{itemize}
        together with the naturally induced boundary coloring.

        \begin{lemma}\label{coloredMerging}
            With the notations above, there exists a regular finite cover $\tilde{M}$
            of $M$, obtained by gluing elevated colored chunks along boundary
            tori matching the coloring. Moreover, $\tilde{M}$ contains at least one
            elevated positively colored $J_0$ piece.
        \end{lemma}

        \begin{proof}
            Because every colored chunk is naturally a semicover of $M$ (Definition \ref{semicover}),
            applying Corollary \ref{virtualExtensionOfSemicovers} and Proposition \ref{mergeFiniteCovers},
            there exists a positive integer $m$, so that
            for each colored chunk there is an elevated colored chunk semicovering $M$ and 
            inducing the $m$-characteristic covering on the   boundary. 
            Suppose $\hat{J}$ is such a cover of $J$ for any
            JSJ piece $J\subset M$, and $\hat{X}^*_i$ is such a cover for any $X^*_i$.
                   
            Suppose $\hat{X}^*_i$ has $k_i$ positively colored boundary components, and 
            hence $k_i$
            negatively colored boundary components. 
            Suppose $\hat{J}_0$ has $l_i$ boundary components
            that are elevations of $T_i$. Let $K$ be the least common multiple of all $k_i$.
            We take $K$ copies of positively colored $\hat{J}_0$, $K$ copies of negatively colored $\hat{J}_0$,
            and $\frac{l_iK}{k_i}$ copies of each $\hat{X}^*_i$. Then for each $i=1,\ldots,s$, the number of positively colored
            elevations of $T_i$ match from both sides  and the same holds for negatively colored elevations of $T_i$.
            Thus we may glue these copies along their boundary, matching the coloring, and pick one component of the result
            to obtain a semicover $N$ of $M$. 
   
            Note that $\partial N$ is the union of all null colored tori.
            By Corollary \ref{virtualExtensionOfSemicovers},
            there is a finite cover $\tilde{N}$ of $N$ which embeds into a regular finite cover $\tilde{M}$.
            We decompose $\tilde{N}$ by elevations of the elevated colored chunks that compose $N$,
            and regard any JSJ piece of $\tilde{M}$ not contained in $\tilde{N}$ as an elevated null colored
            JSJ piece. Then $\tilde{M}$ is as desired.
        \end{proof}

    \subsection{Constructing the virtual representation}\label{Subsec-virtualRepresentation}
        We use the construction from the previous subsection 
        to find a virtual extension of the representation
        in the assumption of Theorem \ref{virtualExtensionRep}.
As in the assumption of Theorem \ref{virtualExtensionRep}, let
        $M$ be an orientable closed mixed $3$-manifold, and
        suppose
            $$\rho_0:\pi_1(J_0)\to\mathscr{G}$$
        is a representation of the fundamental group of a geometric piece $J_0$
        in a 
        Lie group $\mathscr{G}$, which restricted to each boundary component
        has nontrivial kernel. We also suppose
        that the images of $\rho$ restricted to the boundary components yield
        a class invertible collection of abelian subgroups of $\mathscr{G}$.
        Fix a representative automorphism of a class inversion
            $$\nu:\mathscr{G}\to\mathscr{G}$$
        with respect to this collection.

        Let $T_1,\cdots,
        T_s$ be the components of $\partial J_0$, and let
            $$\zeta_i\subset T_i$$
        be a slope killed by $\rho_0$, for each $1\leq i\leq s$.
        By the construction of Subsection \ref{Subsec-coloredMerging},
        there is a regular finite cover
            $$\kappa:\,\tilde{M}\to M,$$
        obtained by gluing elevated colored chunks matching the coloring (Lemma \ref{coloredMerging}).

        \begin{lemma}\label{glueRep}
            With the notations above, there is a representation
                $$\tilde\rho:\,\pi_1(\tilde{M})\to\mathscr{G},$$
            satisfying the following:
            \begin{itemize}
            \item for each elevated positively colored $J_0$ piece $\tilde{J}_0\subset\tilde{M}$,
            $\tilde\rho$ restricted to $\pi_1(\tilde{J}_0)$ is conjugate to $\kappa^*(\rho_0)$;
            \item for each elevated negatively colored $J_0$ piece $\tilde{J}_0\subset\tilde{M}$,
            $\tilde\rho$ restricted to $\pi_1(\tilde{J}_0)$ is conjugate to $\nu\circ\kappa^*(\rho_0)$;
            \item for each elevated colored corridor chunk $\tilde{X}^*_i\subset\tilde{M}$,
            $\tilde\rho$ restricted to $\pi_1(\tilde{X}^*_i)$ is cyclic; and
            \item for any elevated null colored JSJ piece of $\tilde{J}\subset\tilde{M}$,
            $\tilde{\rho}$ restricted to $\pi_1(\tilde{J})$ is trivial.
            \end{itemize}
        \end{lemma}

        \begin{proof}    By  Lemma \ref{cut-paste} (see also Remark \ref{also-cut-paste}), 
        we need only to construct the local representations with given properties so that they  
        agree on each boundary component up to congugcy.         
       
            In the statement of Lemma \ref{glueRep}, the representations restricted to elevated
            colored $J_0$ pieces and to elevated null colored pieces
            describe themselves. We explain the representation for
            elevated colored corridor chunks as follows.

            Suppose $(X^*_i, R^*_i)$ is a corridor chunk associated to
            $(T_i,\zeta_i)$. We write the canonical semicovering from $X^*_i$
            to $M$ as
                $$\mu_i:\,X^*_i\to M.$$
            Remember that in Subsection \ref{Subsec-coloredMerging},
            we have fixed an orientation of the corridor surface $R^*_i$ for convenience.
            The oriented properly embedded subsurface ${R}^*_i$
            represents a class $[R^*_i]\in H_2(X^*_i,\partial X^*_i;\ZZ)
            \cong H^1(X^*_i;\ZZ)$.
            Then homological pairing with $[R^*_i]$
            induces a quotient homomorphism
                $$\phi_i:\,
                \pi_1(X^*_i)\longrightarrow H_1(X^*_i;\,\ZZ)
                \stackrel{[R^*_i]}{\longrightarrow} \ZZ.$$
            For any positively or negatively
            colored elevation $T^*_i\subset\partial X^*_i$
            of $T_i$, since $R^*_i$ meets $T^*_i$ in exactly one slope,
            $\phi_i$ surjects into $\ZZ$ when restricted to $\pi_1(T^*_i)$.
            Suppose $\gamma^*_i\subset T^*_i$ is a directed slope so that
            $\phi_i([\gamma^*_i])$ equals $1$ in $\ZZ$. We define
            a representation $\alpha_{T^*_i}:\,\ZZ\to\mathscr{G},$
            by assigning $\alpha_{T^*_i}(1)$ to
            be either
            $(\rho_0\circ(\mu_i)_\sharp)([\gamma^*_i])$ or $(\nu\circ\rho_0\circ(\mu_i)_\sharp)([\gamma^*_i])$,
            according to $T^*_i$ being positively or negatively colored, respectively.
            Note that $\alpha_{T^*_i}$ is well defined up to conjugacy of $\mathscr{G}$.
            By the construction of $(X^*_i,R^*_i)$, and in the notations
            of Subsection \ref{Subsec-coloredMerging},
            $X^*_i$ is the carrier chunk of $R^*_i$
            in a regular finite cover $X''_i$
            of the carrier chunk $X'(R'_i)\subset M'_i$, so
            any two $T^*_i$'s differ only by a deck transformation
            of $X''_i$ over $X'(R'_i)$. It follows that $\alpha_{T^*_i}$ up to conjugacy
            is independent of the choice of $T^*_i$. In other words, we have a representation:
                $$\alpha_i:\,\ZZ\to\mathscr{G},$$
            defined by any $\alpha_{T^*_i}$.
            We define
                $$\rho_i:\,\pi_1(X^*_i)\to\mathscr{G}$$
            as $\alpha_i\circ\phi_i$, up to conjugacy.
            Finally, for an elevated corridor chunk $\tilde{X}^*_i\subset\tilde{M}$,
            with the defining covering
                $$\tilde\kappa_i:\,\tilde{X}^*_i\to X^*_i,$$
            we define
                $$\tilde\rho:\,\pi_1(\tilde{X}^*_i)\to\mathscr{G}$$
            as $\rho_i\circ(\tilde\kappa_i)_\sharp$, up to conjugacy.

            We must check that
            the representation $\tilde\rho|_{\pi_1(\tilde{X}^*_i)}$
            agrees with the  adjacent representations up to conjugacy along the boundary.
            Note that $\tilde{X}^*_i$ is only adjacent to
            elevated null colored JSJ pieces and elevated colored $J_0$
            pieces.
            If $\tilde{T}\subset\partial\tilde{X}^*_i$ is
            null colored, this means that under $\kappa_i$,
            $\tilde{T}$ covers a boundary torus
            of $X^*_i$ that misses $\partial R^*_i$. Then $\tilde\rho$ is
            trivial restricted to $\tilde{T}$, and it agrees with the trivial
            representation $\tilde\rho$ on the adjacent elevated null colored
            piece. If $\tilde{T}\subset\partial\tilde{X}^*_i$ is positively
            or negatively colored, it follows from the definition of $\alpha_i$
            that $\rho_i$ restricted to each $\pi_1(T^*_i)$ is conjugate to
            the restriction of $\rho_0$ or $\nu\circ\rho_0$ according to the coloring.
            Thus $\tilde\rho|_{\pi_1(\tilde{X}^*_i)}$ is also
            conjugate to the restriction
            of $\rho_0$ or $\nu\circ\rho_0$ to $\pi_1(\tilde{T})$ according to the
            coloring, since $\tilde\rho|_{\pi_1(\tilde{X}^*_i)}$ is the pull back
            of $\rho_i$ via the subgroup inclusion $(\tilde\kappa_i)_\sharp$.

            Because every elevated positively or negatively colored $J_0$ piece is
            only adjacent to corridor chunks, and because $\tilde\rho$ trivially
            agrees along a torus adjacent to two elevated null colored pieces,
            we have verified that the $\tilde\rho$  we have defined on the elevated
            color chunks of $\tilde{M}$ agree up to conjugacy on the tori that
            they glue up along. We conclude 
            that there is a representation $\tilde{\rho}:\pi_1(\tilde{M})\to\mathscr{G}$, as desired.
        \end{proof}

		Lemma \ref{glueRep} implies Theorem \ref{virtualExtensionRep}, so we have completed the proof
		of Theorem \ref{virtualExtensionRep}.
		
	\begin{remark} \label{also-cut-paste}	 
           
            Note our  $\pi_1(\tilde{M})$ isomorphic the fundamental group of
            the graph-of-groups induced by the obvious graph-of-spaces decomposition, canonical
            up to choosing base points of vertex spaces and paths to base points of adjacent
            edge spaces, and up to choosing a base point of $\tilde{M}$ and
            paths to the base points of vertex spaces, cf.~\cite{Serre}.

            In general, we can glue up representations on vertex groups as long as they agree
            on the edge groups up to conjugacy. This is a consequence of the following
            facts. If $\Gamma=\Gamma_1*_H\Gamma_2$ is an amalgamation of groups, and
            if $\rho_i:\Gamma_i\to\mathscr{G}$, where $i=1,2$,
            are representations such that $\rho_1|_H$ are conjugate to $\rho_2|_H$, then there is
            a representation $\rho:\Gamma\to\mathscr{G}$. More precisely, suppose $\rho_1|_H=\sigma_h\circ\rho_2|_H$,
            where $\sigma_h$ is the conjugation of $h\in\mathscr{G}$, then $\rho$ can be defined by
            taking $\rho_1$ on $\Gamma_1$ and $\sigma_h\circ\rho_2$ on $\Gamma_2$. Similarly,
            if $\Gamma=\Gamma_0*_H$ is an HNN extension with stable letter $t$, and
            if $\rho_0:\Gamma_0\to\mathscr{G}$ is a representation such that $\rho_0|_H$
            is conjugate to $\rho_0|_{H^t}$, say by $\sigma_h$, then there is a representation
            $\rho:\Gamma\to\mathscr{G}$, for example, defined by taking $\rho_0$ on $\Gamma_0$ and
            $\rho(t)=h$.
            \end{remark}

\section{Volume computation}

	In this section, we prove Theorem \ref{virt-non-zero} and Proposition \ref{Seifert volumes of Seifert manifolds}
	using the techniques developed in the previous sections.

	\subsection{Virtually positive volume of representations}  We apply Theorem \ref{virtualExtensionRep} to
	prove Theorem \ref{virt-non-zero}.
%
%
        The lemma below verifies class inversion properties of
        ${\rm PSL}(2;\C)$ and ${\rm Iso}_e\t{\rm SL_2(\R)}$.
        Moreover, for the discussion about representation volumes,
        we are particularly interested in whether the class inversions
        can be realized by conjugation using 
        orientation preserving isomorphisms of the geometric space.

        \begin{lemma}\label{geometricClassInv}
            \ 
            \begin{enumerate}
            \item ${\rm PSL}(2;\C)$ is
            class invertible with respect to all its cyclic subgroups,
            and a class inversion can be
            realized by an inner automorphism of ${\rm PSL}(2;\C)$,
            which is orientation preserving acting on $\Hi^3$;
            \item ${\rm Iso}_e\t{\rm SL_2(\R)}$ is class invertible with
            respect to its center $\RR$, and a class inversion can
            be realized by an inner automorphism of ${\rm Iso}\t{{\rm SL}_2(\R)}$,
            which is orientation preserving acting on $\t{{\rm SL}_2(\R)}$.
            \end{enumerate}
        \end{lemma}

        \begin{proof}
            The first statement follows from the fact
            that every element of $\mobgp$ is conjugate to its inverse in $\mobgp$.
            To see the second statement, note that ${\rm Iso}\t{\rm SL_2(\R)}$ has two components.
            For any $\nu$ in the
            non-identity component, conjugating ${\rm Iso}_e\t{\rm SL_2(\R)}$ by $\nu$ sends
            any $r\in\RR$ to $-r\in\RR$, so it is a class inversion for $\RR$.
            Recall that there are no orientation reversing isometries in the ${\rm SL}_2$-geometry. 
        \end{proof}

        \begin{proof}[{Proof of Theorem \ref{virt-non-zero}}]
            We first show the hyperbolic volume case. Suppose $M$ contains at least one
            hyperbolic piece $J_0$. It suffices to prove the theorem when $M$ is mixed. We take sufficiently long slopes,
            one in each component of $\partial J_0$, making sure that
            Dehn fillings along these slopes yield a closed hyperbolic $3$-manifold $\bar{J}_0$ of finite volume.
            Let $\rho_0:\pi_1(J_0)\to{\rm PSL}(2;\C)$
            be the representation factoring through the the Dehn filling and the discrete faithful representation
            of $\pi_1(\bar{J}_0)$. By Theorem \ref{virtualExtensionRep}, we can virtually extend $\rho_0$
            to $\rho:\pi_1(\tilde{M})\to{\rm PSL}(2;\C)$. Moreover, it follows
            from the conclusion of Theorem \ref{virtualExtensionRep} and
            the additivity principle (Theorem \ref{additivity}) and Lemma \ref{Almost trivial representation} that
            only some elevations of $J_0$ could
            contribute to the hyperbolic representation volume of $\tilde{M}$.
            By Lemma \ref{geometricClassInv}(1) the volume of
            all these elevations is a positive multiple  of the hyperbolic volume of $\bar{J}_0$.
            Thus the hyperbolic representation volume of $\tilde{M}$ is positive.

            It remains to show the Seifert volume case. Since the theorem  is known
            for graph manifolds \cite{DW} and and  for
            geometric manifolds \cite{BG1},
            we may again assume $M$ to be mixed. By the assumption, $M$ also
            contains a Seifert geometric piece $J_0$.
            The rest of the argument is almost the same as the previous case,
            except that: 
            we start by picking a slope $\zeta_i\subset T_i$ which intersects the Seifert fiber $t_i\subset T_i$ exactly once 
            for each component $T_i$ of $\partial J_0$, moreover those $\zeta_i$ can be chosen so that the Dehn filling $\bar{J}_0$ of
            $J_0$ has a nontrivial Euler class. Then we can choose $[t_i]$ to be the $\gamma_i^*$ in the proof of Lemma \ref{glueRep} and 
            applying Theorem \ref{virtualExtensionRep}, Lemma \ref{geometricClassInv}(2),
            and Theorem \ref{additivity},  we will find a finite cover $\tilde{M}$
            with positive Seifert volume.
             \end{proof}

\subsection{Volumes of representations of Seifert manifolds}
Now we  will prove Proposition \ref{Seifert volumes of Seifert manifolds}.

Let $N$ be a closed oriented $\t{{\rm SL}_2(\R)}$-manifold whose
base $2$-orbifold is an orientable,  hyperbolic
$2$-orbifold $\c{O}$ with positive genus $g$ and  $p$ singular points. Then, keeping the same notation as in section 2.3, we have a
presentation
$$\pi_1N=\l \alpha_1,\beta_1,\ldots,\alpha_g,\beta_g,s_1,\ldots,s_p,h :  $$
$$s_1^{a_1}h^{b_1}=1,\ldots, s_p^{a_p}h^{b_p}=1, [\alpha_1,\beta_1]\ldots[\alpha_g,\beta_g]=s_1\ldots s_p\r$$
with the condition $e=\sum_ib_i/a_i\not=0$.
The following result
of Eisenbud--Hirsch--Neumann \cite{EHN}, which extends the result of
Milnor--Wood \cite{Mi,Wo} from circle bundles to Seifert manifolds, is very
useful for our purpose.

\begin{theorem}[{\cite[Theorem 3.2 and Corollary 4.3]{EHN}}]\label{Eisenbud-Hirsch-Neumann} 
	Suppose $N$ is a closed orientable Seifert manifold with
	a regular fiber $h$ and  base of genus $>0$.
	\begin{enumerate}
		\item There is a $({\rm PSL}_2(\R),{\S}^1)$
		horizontal foliation on $N$ if and only if there is a representation
		$\t{\phi} : \pi_1(N)\to \widetilde{{\rm SL}_2(\R)}$ such that
		$\t{\phi} (h)= {\rm sh}(1)$;
		\item 
		Suppose $N=(g,0;a_1/b_1,\ldots, a_n/b_n)$, then there is a 
		$({\rm PSL}_2(\R),{\S}^1)$ horizontal foliation on $N$ if and only if
			$$\sum \llcorner{ {b_i}/{a_i}}\lrcorner \le -\chi(F_g); \,\,\, \sum \ulcorner{b_i}/{a_i}\urcorner \ge\chi(F_g)$$
		
	\end{enumerate}
\end{theorem}

In order to prove Proposition \ref{Seifert volumes of Seifert
manifolds} we will check the following proposition which describes
those representations leading to a non zero volume. For each element
$(a, b)\in {\R}\times\t{{\rm SL}_2(\R)}$, its image in
${\R}\times_{\Z}\t{{\rm SL}_2(\R)}$ will be denoted as $\o{(a, b)}$.

\begin{proposition}\label{SvSm}
A representation $\rho: \pi_1(N)\to {\rm Iso}_e\t{{\rm
SL}_2(\R)}={\R}\times_{\Z}\t{{\rm SL}_2(\R)}$ has non-zero volume if and only if
there are integers $n,n_1,\ldots,n_p$ subject to the conditions
\begin{eqnarray}
\sum \llcorner{ {n_i}/{a_i}}\lrcorner -n\le 2g-2 \,\,\, {\rm and}  \,\,\,\sum \ulcorner{n_i}/{a_i}\urcorner-n \ge
2-2g
\end{eqnarray}
such that
\begin{eqnarray}
\rho(s_i)=\o{\left(\frac{n_i}{a_i}-\frac{b_i}{a_i}\frac{1}{e}\left(\sum_i\left(\frac{n_i}{a_i}\right)-n\right),g_i{\rm sh}\left(\frac{-n_i}{a_i}\right)g^{-1}_i\right)}
\end{eqnarray}
where $g_i$ is an element of $\t{{\rm SL}_2(\R)}$ and
\begin{eqnarray}
\rho(h)=\o{\left(\frac{1}{e}\left(\sum_i\left(\frac{n_i}{a_i}\right)-n\right),1\right)}
\end{eqnarray}
whose volume is given by
\begin{eqnarray}
{\rm
vol}(N,\rho)=4\pi^2\frac{1}{|e|}\left(\sum_i\left(\frac{n_i}{a_i}\right)-n\right)^2
\end{eqnarray}
Moreover the $\rho$-image of $\alpha_1,\beta_1, \ldots ,
\alpha_g,\beta_g$ can be chosen to lie in $\t{{\rm SL}_2(\R)}$.
\end{proposition}

\begin{proof}
The condition ${\rm vol}(N,\rho)\ne 0$ implies that $\rho(h)=
\o{(\zeta, 1)}\in G={\R}\times_{\Z}\t{{\rm SL}_2(\R)}$ by \cite[p.~663]{BG1} and 
\cite[p.~537]{BG2}, using a cohomological-dimension argument and the definition in paragraph 2.2. Suppose $\rho(s_i) = \o{(z_i,
x_i)}$. Then $s_i^{a_i}h^{b_i}=1$ implies that $$\o{({a_i}z_i,
x^{a_i})} \o{({b_i} \zeta,1)}=\o{({a_i}z_i+{b_i}\zeta, x^{a_i})}=1.$$
Then there is an $n_i\in{\Z}$ such that (see Remark \ref{SL})
\begin{eqnarray}
{a_i}z_i+{b_i}\zeta_i \textrm{ in}\, \R \textrm{ and }
x_i \textrm{ is conjugate in $\t{{\rm SL}_2(\R)}$ to } {\rm sh}\left(-\frac{n_i}{a_i}\right).
\end{eqnarray}
Since $[\alpha_1,\beta_1]\ldots[\alpha_g,\beta_g]=s_1\ldots s_p$ and since the
product of commutators in ${\R}\times_{\Z}\t{{\rm SL}_2(\R)}$ must
lie  in $\t{{\rm SL}_2(\R)}$ this implies that
$$\o{(z_1+\ldots+z_p, x_1\ldots x_p)}=\o{\left(0,
\prod_{j=1}^g[\rho(\alpha_j),\rho(\beta_j)]\right)}.$$
Then there is an $n\in{\Z}$ such that
\begin{eqnarray}
z_1+\ldots+z_p \textrm{ and } \prod_{j=1}^g[\rho(\alpha_j),\rho(\beta_j)]=x_1\ldots x_p{\rm sh}(n)
\end{eqnarray}
Equalities  (6.6) and (6.5),  imply
 condition (6.1) in Proposition \ref{SvSm} using  Theorem \ref{Eisenbud-Hirsch-Neumann} and
its proof in \cite{EHN}.
By (6.5) and (6.6), we can calculate directly
\begin{eqnarray}
z_i=\frac{n_i}{a_i}-\frac{b_i}{a_i}\zeta,\,\,\,  \zeta= \frac 1e \left(\sum_{i=1}^p\frac{n_i}{a_i}-n\right)
\end{eqnarray}
Plugging (6.5),  (6.6) and (6.7) into $\rho(h)= (\zeta, 1)$ and $\rho(s_i)
= (z_i, x_i)$, we obtain (6.2) and (6.3) in Proposition
\ref{SvSm}.
Then the  ``moreover'' part of Proposition \ref{SvSm} also follows from Theorem
\ref{Eisenbud-Hirsch-Neumann}.

Let's now compute the volume of such a representation. Let $p_1: \t
N \to N$ be a covering from  a circle bundle $\t N$ over $\t F$ to
$N$ so that the fiber degree is 1.  Then we have $$\t e=e(\t
N)=(\mathrm{deg} p_1) e.$$ Let $\t t$ be the fiber of $\t N$ and $\t
\rho =\rho|\pi_1\t N$. Then $\left(\t{t}\right) ^{\t e}=\prod_{j=1}^{\t
g}[\t{\alpha_j},\t{\beta_j}]$ in $\pi_1\t N$, and therefore $\t \rho
(\left(\t{t}\right) ^{\t e})= (\o{\t e\zeta, 1})\in Z(G)\cap \t{{\rm
SL}_2(\R)}$, since the image of the fiber must be in the center and
the image of the product of commutators must lie in $\t{{\rm
SL}_2(\R)}$. Hence $\t e \zeta= \t n \in{\Z}$.

Let $p_2: \hat N\to \t N$ be the covering along the fiber direction
of degree $\t e$, and then $\hat e = e(\hat N)=1$. Then $\hat \rho =
\t \rho|$ sends actually $\pi_1\hat N$  into  $\t{{\rm SL}_2(\R)}$
and the fibre $\hat t$ of $\hat N$ is sent to ${\rm sh}(\t{n})$. Finally there is a covering $p_*: \hat N\to N^*$ along the fiber
direction of degree $\t n$,  where $N^*$ is a circle bundle over
a hyperbolic surface $F$ with $e^*=e(N^*)=\t n$. It is apparent that
$\hat \rho$ descends to $ \rho^*: \pi_1 N^*\to \t{{\rm SL}_2(\R)}$
such that $\rho^*(h^*)=\rm sh(1)$, where $h^*$ denotes the
${\S}^1$-fiber of $N^*$. According to Theorem
\ref{Eisenbud-Hirsch-Neumann}, there is $({\rm PSL}_2(\R),{\S}^1)$-horizontal foliation on $N^*$, and according to Proposition
\ref{Euler}, ${\rm vol}(N^*,\rho^*)=4\pi^2e^*=4\pi^2\t n$, and then
$${\rm vol}(\hat{N},\hat \rho)=4\pi^2\t n^2=4\pi^2\t e^2\zeta^2.$$
Note that $$\mathrm{deg} p_1 \mathrm{deg} p_2= \frac{\t e}  e\times \t e=
\frac {\t e^2} e.$$ By those facts  we reach (6.4) as
below:
$${\rm vol}(N,\rho)=\frac{{\rm vol}(\hat{N},\hat \rho)}{\mathrm{deg} p_1 \mathrm{deg} p_2}=\frac{4\pi^2\t e^2\zeta^2}{\frac {\t e^2} e}= 4\pi^2 {e}{\zeta^2}=\frac{4\pi^2}{|e|}\left(\sum_{i=1}^p\frac{n_i}{a_i}-n\right)^2.$$
\end{proof}

\begin{remark} Suppose in Proposition \ref{SvSm} that $n_i=a_ik_i+r_i$, where $0\le r_i<a_i$. If we choose $n=2-2g+\sum_ik_i$ and $n_i=(k_i+1)a_i-1$ then
the corresponding representation $\rho_0$ is faithful, discrete and reaches the maximal volume
giving rise to the well known formula
$${\rm vol}(N,\rho_0)=\frac{4\pi^2\chi_{O(N)}^2}{|e(N)|}.$$
\end{remark}

\section{Volumes of representations do not have the covering property}
In this section, we prove Theorem \ref{zero} and therefore Corollary \ref{non-cover}.
They follow immediately from the two propositions of this section.

\subsection{Non-trivial graph manifolds with vanishing Seifert volume}

 \begin{proposition}\label{zero-1}
 There are infinitely many non-trivial graph manifolds with zero Seifert volume.
\end{proposition}
We begin with an elementary lemma in $\t{{\rm SL}_2(\R)}$ geometry. Recall Lemma 2.3.

\begin{lemma}\label{trivial}
Let $\Gamma$ be a subgroup of $G$ and denote by $\o{\Gamma}$ its
projection onto ${\rm PSL}(2;{\R})$. If $\o{\Gamma}$ is abelian then
so is $\Gamma$.
\end{lemma}
\begin{proof}
This follows from \cite[Lemma 2.1]{EHN}. Let $g=\o{(\xi,x)}$ and
$h=\o{(\eta,y)}$ be two elements of $\Gamma$. Then note that
$[g,h]=\o{(0,[x,y])}=[x,y]$  is actually a commutator in $\t{{\rm
SL}_2(\R)}$. If $\o{\Gamma}$ is abelian then $[g,h]$ belongs to
${\R}\cap\t{{\rm SL}_2(\R)}={\Z}$ and there exists an integer $k$
such that $[x,y]={\rm sh}k$.

Recall that $x$ can be seen as a homeomorphism of the real line and
using the notations of \cite{EHN} we set $m(z)=\min_{z\in{\R}}
x(z)-z$ and $M(z)=\max_{z\in{\R}} x(z)-z$. Notice that since $x$ is
a lifting of an orientation preserving homeomorphism of the circle
then $x(z+1)=x(z)+1$ and these $\min$ and $\max$ can be considered
only on $[0,1]$ so that the definition makes sense.

Besides  by \cite[Lemma 2.1(5)]{EHN}  $\llcorner
m(xyx^{-1})\lrcorner=\llcorner m(y)\lrcorner$ and $\ulcorner
M(xyx^{-1})\urcorner=\ulcorner M(y)\urcorner$. Since $xyx^{-1}={\rm
sh}(k).y$ we have by \cite[Lemma 2.1(4)]{EHN} and the first equality
$$k+m(y)\leq m({\rm sh}(k)y)=m(xyx^{-1}),$$ and then
$$k+\llcorner m(y)\lrcorner\leq\llcorner m({\rm sh}(k)y)\lrcorner=\llcorner m(xyx^{-1})\lrcorner=\llcorner m(y)\lrcorner,$$
and by  \cite[Lemma 2.1(4)]{EHN} and the second equality
$$k+M(y)\geq M({\rm sh}(k)y)=M(xyx^{-1}),$$ and then
$$k+\ulcorner M(y)\urcorner\geq\ulcorner M({\rm sh}(k)y)\urcorner=\ulcorner M(xyx^{-1})\urcorner=\ulcorner M(y)\urcorner.$$
This forces $k=0$ and therefore $[g,h]=[x,y]=1$. This proves the
lemma.
\end{proof}
\begin{proof}[Proof of Proposition \ref{zero-1}]
The proof follows from a construction in Motegi \cite{Mo}. We recall it. Let
$(p_1,q_1)$ and $(p_2,q_2)$ be two pairs of co-prime integers and
consider $E_1$ and $E_2$ the orientable Seifert manifolds over a
$2$-disk with two exceptional fibres whose fundamental groups are
given by
$$\pi_1E_1=\l c_1,c_2,t_1, \ [c_1,t_1]=[c_2,t_1]=1, c_1^{p_1}=t_1^{r_1}, c_2^{q_1}=t_1^{s_1}\r,$$
and
$$\pi_1E_2=\l d_1,d_2,t_2, \ [d_1,t_2]=[d_2,t_2]=1, d_1^{p_2}=t_2^{r_2}, d_2^{q_2}=t_2^{s_2}\r.$$
These are the exterior of two torus knots whose meridians are
denoted by $m_1$ and $m_2$. Notice that $E_i$ is Euclidean if and only if
$p_i=q_i=2$ and otherwise it is an ${\Hi}^2\times{\R}$-manifold.
The couples $(m_1,t_1)$ and $(m_2,t_2)$ provide a basis of $H_1(\b
E_1;{\Z})$ and of $H_1(\b E_2;{\Z})$ and Motegi constucted a closed
graph manifold $M$ from $E_1$ and $E_2$  via an orientation
reversing identification $\varphi\co\b E_1\to\b E_2$ sending $t_1$
to $m_2$ and $m_1$ to $t_2$.

In \cite[Section 3]{Mo} Motegi checked that $H_1(M;{\Z})$ is
isomorphic to $\z{(p_1p_2q_1q_2-1)}$ and that any representation of
$\pi_1M$ into ${\rm PSL}(2;{\C})$ is abelian. Hence for any representation  $\rho\co\pi_1M\to G$, $\o{\rho(\pi_1M)} \subset  {\rm PSL}(2;{\R})\subset  {\rm PSL}(2;{\C})$ must be abelian.
By Lemma \ref{trivial}  $\rho(\pi_1M)\subset  G$ must be
abelian and since $H_1(M;{\Z})$ is finite then so is the image $\rho(\pi_1M)\subset  G$. 
This proves ${\mathrm{SV}}(M)=0$ by Lemma \ref{Almost trivial representation}. To complete the proof of
the proposition notice that $M$ is a non-trivial graph manifold if and only if
$|H_1(M;{\Z})|=p_1p_2q_1q_2-1>15$.
\end{proof}

\subsection{Mixed 3-manifolds with vanishing hyperbolic volume}

\begin{proposition}\label{zero-2} There are infinitely many  3-manifolds
$N$   with non-vanishing $||N||$ but ${\rm vol}(N,{\rm
PSL}(2;{\C}))=\{0\}$.
\end{proposition}

\begin{proof} We first
begin by constructing a closed mixed 3-manifold with one hyperbolic piece
adjacent to one  Seifert piece whose hyperbolic volume vanishes.  Let
$M_1$ denote $F\times{\S}^1$ where $F$ is a surface with positive
genus and connected boundary. There is a natural section-fiber basis
$(s,h)\subset \partial M_1$. On the other hand, it follows  from
\cite{HM} that  there are infinitely many one cusped, complete,
finite volume hyperbolic manifolds $M_2$ endowed with a basis
$(\mu,\lambda)\subset \partial M_2$ such that both $M_2(\lambda)$
and $M_2(\mu)$ have zero simplicial volume (because they are
actually connected sums of lens spaces). Denote by $\varphi\co\b
M_1\to \b M_2$ the homeomorphism defined by $\varphi(s)=\mu$ and
$\varphi(h)=\lambda^{-1}$. Let $M_{\varphi}=M_1\cup_{\varphi}M_2$.
Then $M_{\varphi}$ is a mixed manifold. Denote $\c{T}_{M_{\varphi}}$ by $T$.

Let  $\rho\co\pi_1M_{\varphi}\to{\rm PSL}(2;{\C})$ be any
representation and denote by $A$ the resulting connection over
$M_{\varphi}$. Notice that either $\rho(s)$ or $\rho(h)$ is trivial.
Indeed  if $\rho(h)\not=1$, its centralizer $Z(\rho(h))$ must be
abelian in ${\rm PSL}(2;{\C})$. Since $h$ is central in $\pi_1M_1$,
this means that  $\rho(\pi_1M_1)$ is abelian. Since $s$ is
homologically   zero in $M_1$,   then $\rho(s)=1$.

Let $\zeta$ be either $s$ or $h$ so that  $\rho(\zeta)=1$. After
putting $A$ in normal form with respect to $T$, denote by $A_1$ and
$A_2$ the flat connections over $M_1$ and $M_2$ respectively. Since
$\rho(\zeta)$  is trivial then $A_1$ and $A_2$ do extend over
$M_1(\zeta)$ and $M_2(\zeta)$ to flat connections $\hat{A}_1$ and
$\hat{A}_2$,  and thus
$$\mathfrak{cs}_{M_{\varphi}}(A)=\mathfrak{cs}_{{M}_1(\zeta)}(\hat{A}_1)+\mathfrak{cs}_{{M}_2(\zeta)}(\hat{A}_2).$$
Eventually taking the imaginary part we get
\begin{eqnarray}
{\rm vol}(M_\varphi,\rho)={\rm vol}({M}_1(\zeta),\hat{\rho}_1)+ {\rm
vol}({M}_2(\zeta),\hat{\rho}_2)
\end{eqnarray}
 where $\hat{\rho}_i$
denotes the extension of $\rho|\pi_1M_i$ to $\pi_1M_i(\zeta)$. Since
both ${\rm vol}({M}_1(\zeta),\hat{\rho}_1)$ and  ${\rm
vol}({M}_2(\zeta),\hat{\rho}_2)$ do vanish,  the proof of
Proposition \ref{zero-2} is complete.
\end{proof}



\section{Conclusions}

	In conclusion, given a geometrically meaningful representation
	of a $3$-manifold group, 
	Chern--Simons theory can be applied to compute the associated volume. 
	On the other hand, recent results about separability of surface subgroups
	are powerful in constructing interesting
	virtual representations of $3$-manifold groups.
	However, a shortcoming of our approach seems to be that we are not
	able to control the degree of the cover that we need to pass to,
	so, for instance, we do not have lower bound 
	estimations of the growth of virtual volumes of representations.
	
	We propose two further problems.
	
	\begin{problem}
		Estimate the growth of virtual hyperbolic volume and virtual Seifert volume.
	\end{problem}
	
	Since the hyperbolic volume is bounded by the simplicial volume, it has at most linear growth
	as ${\mathrm{HV}}(\tilde{M})\,/\,[\tilde{M}:M]$ is bounded by $\mu_3||M||$. However,
	we do not know whether Seifert volume has at most linear growth as well.
	
	\begin{problem}
		Is the Seifert volume of a closed prime $3$-manifold 
		virtually positive if it has positive simplicial volume?
	\end{problem}
	
	The open case is when the $3$-manifold has only hyperbolic pieces 
	in its geometric decomposition.

\end{document}